\documentclass{article}
\usepackage{amsmath,amssymb,amsthm}
\usepackage{graphicx}
\usepackage{fancyhdr}
\usepackage{color}
\usepackage{hyperref}

\newcommand{\pd}{\partial}
\def\mbR{\mathbb{R}}
\def\mmm{\mathcal}
\def\eoc{{\rm EOC}}
\newcommand{\uhat}{\hat{u}_{h}}
\newcommand{\uhatx}{\hat{u}_{hx}}
\newcommand{\uhatt}{\hat{u}_{ht}}
\newcommand{\uhattx}{\hat{u}_{htx}}
\newcommand{\what}{\hat{w}_{h}}
\newcommand{\whatx}{\hat{w}_{hx}}
\newcommand{\whatt}{\hat{w}_{ht}}
\newcommand{\Qhatt}{\hat{Q}_{ht}}
\newcommand{\Qhat}{\hat{Q}_{h}}
\def\intertime{\bar{t}}
\def\embconst{\hat{C}(I)}
\def\ubc{u_{b}} 
\def\ubch{I_{h}(u_{b})} 
\def\uic{u_{0}} 
\def\uicx{u_{0x}} 
\def\cic{c_{0}} 
\def\uich{u_{0h}} 
\def\uichx{u_{0hx}} 
\def\cich{c_{0h}} 
\def\wich{w_{0h}}

\def\eps{\varepsilon}

\newtheorem{cor}{Corollary}[section]

\newtheorem{lemma}[cor]{Lemma}
\newtheorem{theorem}[cor]{Theorem}

\newtheorem{remark}[cor]{Remark}
\newtheorem{prob}[cor]{Problem}
\newtheorem{ass}[cor]{Assumption}

\numberwithin{equation}{section}

\begin{document}

\title{Elastic flow interacting with a lateral diffusion process: The one-dimensional graph case}

\author{%
{\sc
Paola Pozzi\thanks{Email: paola.pozzi@uni-due.de}
} \\[2pt]
Universit\"at Duisburg-Essen, Fakult\"at f\"ur Mathematik,\\
Thea-Leymann-Stra{\ss}e 9, 45127 Essen, Germany \\[6pt]
{\sc and}\\[6pt]
{\sc 
Bj\"{o}rn Stinner\thanks{Corresponding author. Email: bjorn.stinner@warwickac.uk}
} \\[2pt]
Mathematics Institute, University of Warwick, \\
Zeeman Building, Coventry CV4 7AL, United Kingdom
}

\maketitle

\begin{abstract}
{
A finite element approach to the elastic flow of a curve coupled with a diffusion equation on the curve is analysed. Considering the graph case, the problem is weakly formulated and approximated with continuous linear finite elements, which is enabled thanks to second-order operator splitting. The error analysis builds up on previous results for the elastic flow. To obtain an error estimate for the quantity on the curve a better control of the velocity is required. For this purpose, a penalty approach is employed and then combined with a generalised Gronwall lemma. Numerical simulations support the theoretical convergence results. Further numerical experiments indicate stability beyond the parameter regime with respect to the penalty term which is covered by the theory. 
}
{

\medskip


\medskip

}
\end{abstract}

\section{Introduction}

The objective of this article is the convergence analysis of a semi-discrete finite element approximation to the following problem: 

\begin{prob} \label{prob:cont}
Given a spatial interval $I:=(0,1)$ and a time interval $(0,T)$ with some $T>0$ and some functions $f : \mbR \to \mbR$, $\uic,\cic : \overline{I} \to \mbR$, and $\ubc : \pd I \to \mbR$, find functions $u,c : I \times (0,T) \to \mbR$ such that
\begin{align} \label{eq1}
\frac{u_{t}}{Q} &= -\frac{1}{Q} \left ( \frac{\kappa_{x}}{Q} \right)_{x} - \frac{1}{2} \kappa^3 + f(c), \\
\label{eq2}
\kappa &= \left( \frac{u_{x}}{Q} \right)_{x}, \\ 
\label{eq3}
(c Q)_{t} &= \left ( \frac{c_{x}}{Q} \right)_{x},
\end{align}
where
\[
 Q(x,t) := \sqrt{1+ u_{x}^{2}(x,t)}, \quad (x,t) \in I \times (0,T),
\]
with the boundary and initial conditions
\begin{align} 
\label{bcond}
& u(x,t) = \ubc(x), \quad \kappa(x,t) = 0, \quad c(x,t) = 0, & \qquad & (x,t) \in \partial I \times [0,T], \\
\label{icond}
& c(x,0) = \cic(x), \quad u(x,0) = \uic(x), & \qquad & x \in \overline{I}.
\end{align}
\end{prob}

The equations \eqref{eq1} and \eqref{eq2} are the graph formulation of the elastic flow for the curve $\{ \Gamma(t) \}_{t \in (0,T)}$,  $\Gamma(t) := \{ (x,u(x,t)) \, | \, x \in I \}$, with a forcing term $f(c)$ in the direction normal to the curve. This term depends on a conserved field $c$ on the curve which is subject to the advection-diffusion equation \eqref{eq3}. Such type of problems are motivated by applications in soft matter, see \cite{ES10b,ES13,MerMarRicHar2013}, and cell biology (\cite{ChaGanGra2001,NeiMacWebIns2011,ESV12}). 

Numerical methods for solving forth-order geometric equations such as \eqref{eq1}, \eqref{eq2} may be based on parametric approaches. This work builds up on the graph formulation of the elastic flow (or Willmore flow for higher dimensional manifolds) and on the results which are presented in \cite{DD06,DD06Corr}. More general parametric methods for the above or related problems are presented and analysed in \cite{DziKuwSch2002,DecDzi2009,BarGarNue2007,BarGarNue2012} for curves and \cite{BaePedNoc2004,Dzi2008,Poz2015} for surfaces. Often, operator splitting is employed, thus enabling the use of $H^1$ conforming spaces. But also more direct approaches exist, for instance, using finite volume techniques as in \cite{MikSevBal2010}, employing methods from isogeometric analysis (\cite{BarDedQua2016}), or using $C^1$ conforming finite elements as in \cite{DecSch2010,DecKatSch2015}. Alternatively, methods may also be based on interface capturing approaches. This includes level set representations of the curve or surface (\cite{OshSet1988,DroRum2004}, see \cite{BenMikObeSev2009} for a comparison with parametric methods) and the phase field methodology (\cite{DuLiuWan2004,DuWan2007,BreMasOud2015,FraRumWir2013}. For an overview we refer to \cite{DDE} but we remark that the field has seen significant advances since. 

The two paradigms of surface representation, parametric approaches versus interface capturing approaches, also underpin techniques for solving PDEs on moving surfaces. The overview by \cite{DziEll2013} lists a variety of methods. These include Lagrange methods using finite elements on triangulated surfaces as in \cite{DE07} or generalised spline representations, see \cite{LanMooNeu2016}, diffuse interface approximations (\cite{RaeVoi2007,EllStiStyWel2011}), or Eulerian approaches based on fixed bulk meshes (\cite{XuZha2003,DziEll2010,OlsReu2014,HanLarZah2016,PetRuu2016}). 

For coupled problems such as \eqref{eq1}--\eqref{eq3} we are not aware of any convergence results. Schemes for curve shortening flow instead of the above elastic flow have been analysed in \cite{PS17} (semi-discrete) and \cite{BarDecSty2017} (fully discrete). The related work of \cite{KovLiLubGue2017} covers the case of a (weighted) $H^1$ flow instead an $L^2$ flow of the surface energy. The benefit then is some additional control of the manifold velocity which allows to show convergence of an isoparametric finite element scheme even in the case of surfaces. 

Our numerical approach to Problem \ref{prob:cont} is based on the method in \cite{DD06,DD06Corr} for the elastic flow of the curve in the graph case. Operator splitting and piecewise linear $H^1$-conforming finite elements are used and, in particular, error estimates for the velocity $u_{t}$, the spatial gradient $u_{x}$, and the length element $Q$ are proved. However, the diffusion equation involves $Q_{t}$, whence some control of $u_{xt}$ is required. Denoting by $h$ the spatial discretisation parameter, the idea is to add a suitably $h$-weighted $H^1$ inner product of the velocity with the test function to the semi-discrete weak problem, see \eqref{def:mu} and \eqref{(2.10)} below. In principle, this idea already features in the scheme in \cite[equation (3.12)]{PS17} where, thanks to mass lumping, such a term with a weighting scaling with $h^2$ is added. For that problem the structure of the geometric equation could be further exploited in order to derive suitable error estimates for $c$. In the present case we use a generalised Gronwall inequality (see Lemma~\ref{genG} below) instead. For this to work we need to assume strictly smaller than quadratic growth in $h$. As a result, we can only prove smaller convergence rates for the geometric fields than in \cite{DD06}. The slower convergence is also observed in numerical simulations. However, the scheme turns out to be quite stable even for faster growth of the penalty term in $h$. In particular, if it grows quadratically in $h$ then we essentially recover the rates in \cite{DD06} (where there is no coupling, i.e., $f=0$).

In Section 2 we state Problem \ref{prob:cont} in a suitable variational form and some assumptions on the continuous solution. The spatial discretisation is presented in Section 3 where we also prove some properties of the semi-discrete scheme and state the main convergence result (Theorem \ref{mainthm} on page \pageref{mainthm}). This result then is proved by a series of Lemmas in Section 4. In Section 5 we present some numerical simulation results and Section 6 contains some concluding remarks.

\section{Variational formulation and assumptions}

Instead of working with the scalar curvature $\kappa$, we introduce the variable 
$$w:=-\kappa Q=-\kappa\sqrt{1+ u_{x}(x,t)^{2}}=-\frac{u_{xx}}{(1+ u_{x}(x,t)^{2})}.$$ 
A simple computation gives
\[
 -\frac{1}{Q} \left( \frac{\kappa_{x}}{Q} \right)_{x} - \frac{1}{2} \kappa^{3} = \left( \frac{1}{Q^{3}} w_{x} \right)_{x} + \frac{1}{2}
\left( \frac{w^{2}}{Q^{3}} u_{x} \right)_{x}.
\]
We thus consider the following weak formulation of the system \eqref{eq1}--\eqref{eq3}: 
\begin{align}
\label{(2.4)}
\int_I \frac{u_{t}}{Q}\varphi \, dx 
+ \int_I \frac{1}{2} w^{2}\frac{u_{x} \varphi_{x}}{Q^{3}} + \frac{w_{x} \varphi_{x}}{Q^{3}}\, dx 
- \int_{I} f(c) \varphi \, dx
\, &= 0 & \quad & \forall \varphi \in H^{1}_{0}(I), \\
\label{(2.5)}
\int_I \frac{w}{Q} \psi\, dx - \int_I \frac{u_{x} }{Q}\psi_{x}\, dx\, &= 0 & \quad & \forall \psi \in H^{1}_{0}(I), \\
\label{c-gleichung}
\frac{d}{dt } \left( \int_{I}c Q \xi \, dx \right)+ \int_{I}\frac{c_{x}}{Q} \xi_{x} \, dx \, &= 0 & \quad & \forall \xi \in H^{1}_{0}(I). 
\end{align}
Note that if we consider a time dependent test function $\xi$ then the last equation is replaced by
\begin{equation} \label{c-gleichung-long}
 \frac{d}{dt } \left( \int_{I}c Q \xi \, dx \right) + \int_{I}\frac{c_{x}}{Q} \xi_{x} \, dx = \int_{I} cQ \xi_{t}. 
\end{equation} 
If $f = 0$ then the system \eqref{(2.4)}, \eqref{(2.5)} coincides with \cite[(2.4), (2.5)]{DD06}.

\begin{ass} \label{ass:sol}
We assume that $f:\mbR \to \mbR$ is a given continuously differentiable map with 
\begin{align}\label{fbound}
\|f\|_{L^{\infty}(\mbR)} \leq C, \qquad \|f'\|_{L^{\infty}(\mbR)} \leq C.
\end{align}
Moreover, we assume that the initial-boundary value problem \eqref{eq1}--\eqref{icond} has a unique solution $(u,c)$ which satisfies
\begin{align}\label{(1.13)}
& u \in L^{\infty}((0,T); W^{4, \infty}(I)) \cap L^{2}((0,T); H^{5}(I)),\\
\label{(1.14)}
& u_{t} \in L^{\infty}((0,T); W^{2, \infty}(I)) \cap L^{2}((0,T); H^{3}(I)),\\
\label{(1.15)}
& u_{tt} \in L^{\infty}((0,T); L^{ \infty}(I)) \cap L^{2}((0,T); H^{1}(I)),\\
\label{cond-c}
& c \in W^{1,\infty}((0,T); H^{1}(I)) \cap L^{\infty}((0,T); H^{2}(I)) \cap L^{\infty}((0,T); H^{1}_{0}(I)) .
\end{align}
\end{ass}

\section{Discretisation and convergence statements}

We consider continuous, piecewise linear finite elements on a subdivision $0 = x_0 < x_1 < \dots <x_N = 1$ of the spatial interval:
\begin{equation*}
X_{h0} := \{ u_{h} \in C^0([0,1], \mbR) \, : \, u_{h}|_{[x_{j-1},x_j]} \in P_1 ([x_{j-1},x_j]),\, j=1 \cdots, N, \, u_{h}(x_0) = u_{h}(x_N) = 0 \}.
\end{equation*}
Let $\varphi_j$, $j=0, \ldots, N$, denote the nodal basis functions. We set $X_{h}:= span \{ \varphi_{0}, \ldots, \varphi_{N} \}$ and denote by $S_{j}$ the subinterval $S_{j}=[x_{j-1}, x_{j}] \subset [0,1] $. Moreover let $h_{j}=|S_j|$ and $h=\max_{j=1,\ldots, N}h_j$ be the maximal diameter of a grid element.
We assume that for some constant $\bar{C}>0$ we have 
\begin{align}
\label{(4.1)}
h_j \geq \bar{C}h \quad \mbox{for all } j=1, \dots, N.
\end{align}
For a continuous function $u \in C^0 ([0,1], \mbR)$ let $I_{h} u \in X_{h}$ be the linear interpolate uniquely defined by $I_{h} u(x_i)= u(x_i)$ for all $i=0, \ldots,N$. We shall use the standard interpolation estimates:
\begin{align}
\label{(4.2)}
\| v-I_{h} v \|_{L^{2}(I)} &\leq C h^k \|v \|_{H^k(I)} &\text{ for $k=1,2$}\,,\\
\label{(4.2)bis}
\| (v-I_{h} v)_{x} \|_{L^{2} (I)} & \leq C h \| v \|_{H^{2} (I)}\,.
\end{align} 
Recall also the inverse estimates for any $m_{h} \in X_{h}$ and $j=1, \ldots, N$:
\begin{align}
\| m_{hx} \|_{L^{2}(S_{j})} & \leq \frac{C}{h_{j}} \| m_{h} \|_{L^{2} (S_{j})} & \quad \overset{\eqref{(4.1)}}{\Longrightarrow} \quad \|m_{hx} \|_{L^{2}(I)} & \leq \frac{C}{h} \| m_{h} \|_{L^{2} (I)}, \label{IE-1} \\
\| m_{h} \|_{L^{\infty} (S_{j})} & \leq \frac{C}{\sqrt{h_{j}}} \| m_{h} \|_{L^{2}(S_{j}) } & \quad \overset{\eqref{(4.1)}}{\Longrightarrow} \quad \| m_{h} \|_{L^{\infty} (I)} & \leq \frac{C}{\sqrt{h}} \| m_{h} \|_{L^{2}(I)} \label{IE-2}.
\end{align}

The discrete formulation that we propose entails a regularization term weigthed by a positive function depending on the parameter $h$, which is defined by
\begin{equation} \label{def:mu}
 \mu(h) := C_{\mu} h^r \quad \mbox{for some } r \in [1,2) \mbox{ and some } C_{\mu} > 0.
\end{equation}
The reason for introducing this term is motivated below in Remark~\ref{regularization} after introducing the necessary notation. The initial data for the discrete problem are denoted by 
\begin{equation} \label{disicond}
 \uich \in \ubch + X_{h0}, \quad \cich \in X_{h0},
\end{equation}
respectively, and will be specified in \eqref{icondh} below (see also  Lemma~\ref{indata}).

\begin{prob}[Semi-discrete Scheme] \label{prob:semidis}
Find functions $u_{h}(\cdot,t) \in \ubch + X_{h0}$ and $w_{h}(\cdot, t), c_{h}(\cdot,t) \in X_{h0}$, $t\in [0,T]$, of the form
\begin{equation*}
 u_{h}(x,t)= \sum_{j=0}^{N} u_j(t) \varphi_j(x), \quad c_{h}(x,t) = \sum_{j=1}^{N-1} c_j(t) \varphi_j(x), \quad w_{h}(x,t) = \sum_{j=1}^{N-1} w_j(t) \varphi_j(x),
\end{equation*} 
with $u_j(t), c_{j}(t), w_{j}(t) \in \mbR$, $t \in [0,T]$, such that $u_{h}(\cdot,0) = \uich$, $c_{h}(\cdot,0) = \cich$ as defined in \eqref{icondh}, and such that for all $\varphi_{h}, \psi_{h}, \zeta_{h} \in X_{h0}$ 
\begin{align}
 \int_I \mu(h) u_{hxt} \varphi_{hx} + \frac{u_{ht} \varphi_{h}}{Q_{h}} \, dx + \int_I \frac{1}{2} w_{h}^{2} \frac{u_{hx} \varphi_{hx}}{Q_{h}^3} + \frac{w_{hx}\varphi_{hx}}{Q_{h}^3} \, dx &= \int_{I} I_{h} (f(c_{h})) \varphi_{h} \, dx, \label{(2.10)} \\
 \int_I \frac{w_{h} \psi_{h}}{Q_{h}} \, dx - \int_I \frac{u_{hx} \psi_{hx}}{Q_{h}} \, dx \, &= 0, \label{(2.11)} \\
 \frac{d}{dt} \left ( \int_{I} c_{h} Q_{h} \zeta_{h} \, dx \right ) + \int_{I} \frac{c_{hx} \zeta_{hx}}{Q_{h}} \, dx \, &=0. \label{ch-gleichung}
\end{align}
Here, $\mu(h)$ is defined in \eqref{def:mu} and $Q_{h}$ denotes the discrete length element,
\[
 Q_{h}(x,t) := \sqrt{1 + u_{hx}^{2}(x,t)}.
\]
\end{prob}
Note that if we consider a time dependent test function $\zeta_{h}(x,t) = \sum_{j=1}^{N-1} \zeta_j(t) \varphi_j(x)$ in \eqref{ch-gleichung} then the last equation is replaced by
\begin{align} \label{ch-gleichung-long}
 \frac{d}{dt} \left ( \int_{I} c_{h} Q_{h} \zeta_{h} \, dx \right ) + \int_{I} \frac{c_{hx} \zeta_{hx}}{Q_{h}} \, dx \, =\int_{I} c_{h} \zeta_{ht} Q_{h} dx.
\end{align}

\begin{lemma}
The above system \eqref{(2.10)}--\eqref{ch-gleichung} has a unique solution on $[0,\tilde{T}]$ for any \mbox{$0<\tilde{T}<\infty$}.
\end{lemma}

\begin{proof} Fix $h>0$. Local existence on some time interval $[0,T_{h})$ follows from standard ODEs theory. 
Since $u_{h}(t), w_{h}(t), c_{h}(t)$ have values in a finite dimensional space (whose dimension depends on $h$), it is sufficient to bound $(u_{h},w_{h},c_{h})$ in some norm to obtain existence on $[0,\tilde{T}]$. 
Choosing $\varphi_{h}= u_{ht}$ in \eqref{(2.10)}, $\psi_{h}=w_{h}$ in $(\ref{(2.11)})_{t}$ (i.e. in equation \eqref{(2.11)} after differentiation with respect to time), and combining the thus obtained equations gives
\begin{align*}
\int_{I} \mu(h) u_{hxt}^{2} dx + \frac{u_{ht}^{2}}{Q_{h}} dx + \frac{1}{2} \frac{d}{dt} \int_{I} \frac{w_{h}^{2}}{Q_{h}} dx = \int_{I} I_{h}(f(c_{h})) u_{ht} dx \leq C \int_{I} Q_{h} dx + \frac{1}{2}\int_{I} \frac{u_{ht}^{2}}{Q_{h}},
\end{align*}
where for the last inequality we have used the boundedness of $f$ (recall \eqref{fbound}). 
Integration in time gives for any $t' \in [0,T_{h})$
\begin{align*}
 \mu(h) \int_{0}^{t'}\int_{I} u_{hxt}^{2} dx dt+ \frac{1}{2}\int_{0}^{t'}\int_{I} \frac{u_{ht}^{2}}{Q_{h}} dx dt + \frac{1}{2} \int_{I} \frac{w_{h}^{2}}{Q_{h}} dx \leq C(\uich, \wich) + C\int_{0}^{t'}\int_{I}Q_{h} dx dt.
\end{align*}
On the other hand, using \eqref{(2.11)} we observe that 
\begin{align*}
\frac{d}{dt} \int_{I} Q_{h } dx &=\int_{I} w_{h} \frac{u_{ht}}{Q_{h}} \leq \epsilon \int_{I} \frac{u_{ht}^{2}}{Q_{h}} dx + C_{\epsilon} \int_{I} \frac{w_{h}^{2}}{Q_{h}} dx 
\end{align*}
so that integration in time gives
\begin{align*}
\int_{I} Q_{h}(t') dx \leq C(\uich) + \epsilon \int_{0}^{t'}\int_{I} \frac{u_{ht}^{2}}{Q_{h}} dx dt + C_{\epsilon} \int_{0}^{t'}\int_{I} \frac{w_{h}^{2}}{Q_{h}} dx dt \qquad 0\leq t'< T_{h}.
\end{align*}
Combining the above inequalities we obtain
\begin{multline*}
 \mu(h) \int_{0}^{t'}\int_{I} u_{hxt}^{2} dx dt+ \frac{1}{2}\int_{0}^{t'}\int_{I} \frac{u_{ht}^{2}}{Q_{h}} dx dt + \frac{1}{2} \int_{I} \frac{w_{h}^{2}}{Q_{h}} dx + \int_{I} Q_h(t') dx \\
 \leq C + C \epsilon \int_{0}^{t'}\int_{I} \frac{u_{ht}^{2}}{Q_{h}} dx dt + C_{\epsilon} \int_{0}^{t'}\int_{I} \frac{w_{h}^{2}}{Q_{h}} dx dt \qquad 0\leq t'< T_{h},
\end{multline*}
for some constant $C=C(\uich,\wich,\tilde{T})$. Choosing $\epsilon$ appropriately and using a Gronwall argument we infer that
\begin{align*}
 \mu(h) \int_{0}^{t'}\int_{I} u_{hxt}^{2} dx dt +\int_{0}^{t'}\int_{I} \frac{u_{ht}^{2}}{Q_{h}} dx dt + \int_{I} \frac{w_{h}^{2}}{Q_{h}}(t') dx + \int_{I} Q_{h}(t') dx \leq C, \qquad 0\leq t'< T_{h}.
\end{align*}
Since all norms are equivalent in a finite dimensional space, this implies that $Q_{h}(t') \leq C(\uich,\wich,\tilde{T},h) $ uniformly in $[0,1] \times [0, T_{h})$. Uniform bounds for $u_{h}$, $w_{h}$ follow immediately.

If we write down explicitly the ODE system for $\dot{u}_{j}$ then we see that
\begin{align*}
\sum_{j=1}^{N-1} \left(\mu(h) \int_{I} \varphi_{ix} \varphi_{jx} dx+ \int_{I} \frac{\varphi_{i} \varphi_{j}}{Q_{h}} dx \right) \dot{u}_{j} =F_{i}(u_{h},w_{h}, c_{h}) \qquad (i=1, 
\ldots N-1)
\end{align*}
with $|F_{i} |\leq C$ uniformly in time, since $f$ is bounded and since we have uniform bounds on $w_{h}$ and $u_{h}$. The $(N-1) \times (N-1)$ matrix $A $ with real entries $A_{ij}(h, Q_{h}(t))= \int_{I}\mu(h) \varphi_{ix} \varphi_{jx} + \frac{\varphi_{i} \varphi_{j}}{Q_{h}} dx$ is symmetric, tri-diagonal, diagonalizable and positive definite. Its positive eigenvalues depend on $h$ but are uniformly bounded from below with repect to time (since $Q_{h}$ is uniformly bounded from above and below).
For simplicity we show this fact in the special case of a uniform grid and taking $\mu(h)=h$ (the general case is treated in a similar way): for the entries of the matrix $A$ a simple computation gives (using that $1 \leq Q_{h}(t') \leq C$)
\begin{align*}
A_{ii}&=\mu(h)\frac{2}{h}+\int_{I}\frac{\varphi_{i}^{2}}{Q_{h}(t')} dx \in \left [2+\frac{1}{C}\frac{4h}{6}, 2 + \frac{4h}{6} \right ],\\
A_{i i\pm 1} & =-\mu(h)\frac{1}{h}+ \int_{I} \frac{\varphi_{i} \varphi_{i\pm 1}}{Q_{h}(t)} dx \in \left[-1 +\frac{1}{C} \frac{h}{6}, -1 +\frac{h}{6}\right].
\end{align*}
It is well known (Gerschgorin theorem) that the eigenvalues $\lambda(t)$ of $A=A(t')$ are elements of the set
$$ \{ z \in \mbR \, : |z- A_{ii}| \leq |A_{i,i+1}| + |A_{i,i-1}| \}$$
giving that
$$ 5\geq \lambda(t) \geq \frac{h}{C}, \qquad \text{ for } 0\leq t'< T_{h}.$$
In conclusion we are able to infer a uniform bound on the $\dot{u}_{j}$, $j=1, \ldots, N-1$ and hence on $u_{hxt}$ (taking into account that $\dot{u}_{0}=\dot{u}_{N}=0$ due to the boundary conditions).

Next, testing \eqref{ch-gleichung-long} with $\zeta_{h}=c_{h}$ and using the bounds on $u_{h}, u_{ht}$ we infer
\begin{align*}
\frac{d}{dt} &\left ( \int_{I} c_{h}^{2} Q_{h} \, dx \right ) + \int_{I} \frac{c_{hx}^{2}}{Q_{h}} \, dx \, =\int_{I} c_{h} c_{ht} Q_{h} dx=\frac{1}{2}\frac{d}{dt} \left ( \int_{I} c_{h}^{2} Q_{h} \, dx \right ) - \frac{1}{2}\int_{I} c_{h}^{2} \frac{u_{hx}}{Q_{h}} u_{hxt} dx \\
&\leq \frac{1}{2}\frac{d}{dt} \left ( \int_{I} c_{h}^{2} Q_{h} \, dx \right ) + C \int_{I} c_{h}^{2} dx
\leq \frac{1}{2}\frac{d}{dt} \left ( \int_{I} c_{h}^{2} Q_{h} \, dx \right ) + C \int_{I} c_{h}^{2} Q_{h}dx.
\end{align*}
With a Gronwall estimate we get $\|c_{h}(t') \|_{L^{2}(I)} \leq C=C(\uich,\wich,\tilde{T},\cich,h)$ uniformly in $0 \leq t' < T_{h}$. The flow can be now extended up to time $\tilde{T}$. Since $h$ was chosen arbitrarily the claim follows.
\end{proof}

We now state our main result which will be proved in the subsequent section by a series of lemmas:
\begin{theorem}\label{mainthm}
Let $f: \mbR \to \mbR$ satisfy \eqref{fbound}. Assume that \eqref{eq1}--\eqref{icond} has a unique solution $(u,c)$ on the interval $[0,T]$, which satisfies \eqref{(1.13)}--\eqref{cond-c}. Let $(u_{h},c_{h})$ denote the solution of Problem~\ref{prob:semidis}.
Then there is some $h_{0} > 0$ such that for all $h \leq h_{0}$ 
\begin{align*}
 \sup_{0\leq t \leq T} \| (u -u_{h})(t) \|_{L^{2}(I)} + \sup_{0\leq t \leq T} \| (w -w_{h})(t) \|_{L^{2}(I)} & \nonumber \\ 
 + \sup_{0\leq t \leq T} \| (u -u_{h})_{x}(t) \|_{L^{2}(I)} & \leq Ch,\\
 \int_{0}^{T} \| (u -u_{h})_{t}(t) \|_{L^{2}(I)}^{2} dt + 
 \int_{0}^{T} \| (w -w_{h})_{x}(t) \|_{L^{2}(I)}^{2} dt & \leq Ch^{2},\\
 \sup_{0\leq t \leq T} \| (c -c_{h})(t) \|_{L^{2}(I)}^{2} + \int_{0}^{T} \| (c -c_{h})_{x}(t) \|_{L^{2}(I)}^{2} dt & \leq Ch^{2}.
\end{align*}
Moreover, we have that
\begin{align*}
\int_{0}^{T} \| (u -u_{h})_{tx}(t) \|_{L^{2}(I)}^{2} dt \leq C \frac{h^{2}}{\mu(h)} = C h^{2-r}.
\end{align*}
with $\mu(h)$ defined in \eqref{def:mu}.
\end{theorem}

\section{Error estimates}
\subsection{Nonlinear Ritz projections}

Our error analysis relies strongly on results presented in \cite{DD06}, which are based on suitable nonlinear Ritz projections for $u$ and $w$. We recall here their definition and properties.
Let $\uhat$ be defined by: $\uhat-I_{h}(\ubc) \in X_{h0}$ and
\begin{align}
\label{def-uhat}
\int_{I} \frac{\uhatx \xi_{hx}}{\Qhat} = \int_{I} \frac{u_{x} \xi_{hx}}{Q} \qquad \forall \xi_{h} \in X_{h0},
\end{align}
\[
 \Qhat(x,t) := \sqrt{1 + \uhatx^{2}(x,t)}.
\]
Note that time $t$ here is a parameter only. For the error
\begin{align*}
\rho_{u} := u -\uhat
\end{align*}
we have the following estimates (see \cite[\S~2]{DD06} and references given in there; to simplify notation we write $\rho_{ux}$ for $(\rho_{u})_{x}$ and so on):
\begin{align}
\label{(2.18)}
& \sup_{0 \leq t \leq T} \| \rho_{u}(t) \|_{L^{2}(I)} + h \sup_{0 \leq t \leq T} \| \rho_{ux}(t) \|_{L^{2}(I)} \leq Ch^{2}, \\
\label{(2.19)}
 & \sup_{0 \leq t \leq T} \| \rho_{u}(t) \|_{L^{\infty}(I)}+ h \sup_{0 \leq t \leq T} \| \rho_{u x}(t) \|_{L^{\infty}(I)} \leq Ch^{2}|\log h| ,\\
 \label{(2.20)}
& \sup_{0 \leq t \leq T} \| \rho_{ut}(t) \|_{L^{2}(I)} \leq Ch^{2}|\log h|^{2}, \\
\label{(2.21)}
& \sup_{0 \leq t \leq T} \| \rho_{utx}(t) \|_{L^{2}(I)} \leq Ch.
\end{align}
We also define a projection $\what \in X_{h0}$ of $w$ with the help of $\uhat$ as follows:
\begin{align}\label{(2.22)}
\int_{I} E(\uhatx)\whatx \varphi_{hx} dx = \int_{I} E(u_{x}) w_{x} \varphi_{hx} dx + \frac{1}{2} \int_{I} w^{2} \left( \frac{u_{x}}{Q^{3}} -\frac{\uhatx}{\Qhat^{3}}\right) \varphi_{hx} dx \quad \forall \, \varphi_{h} \in X_{h0}
\end{align}
where we set
\begin{align} \label{(2.1)}
E(p):= \frac{1}{(1+ p^{2})^{\frac{3}{2}}} \qquad \text{ for } p \in \mbR.
\end{align}
Note that there is some constant $C>0$ such that $|E(p)-E(q)| \leq C |p-q|$ for all $p,q \in \mbR$. The proof of the following bounds for the error
\begin{align*}
\rho_{w} := w- \what
\end{align*}
is given in \cite[Appendix, Lemma~A.1]{DD06}: 
\begin{align}
\label{(2.23)}
& \sup_{0 \leq t \leq T} \| \rho_{w x}(t) \|_{L^{2}(I)} \leq Ch,\\
\label{(2.24)}
& \sup_{0 \leq t \leq T} \| \rho_{w }(t) \|_{L^{2}(I)} \leq Ch^{2} |\log h|,\\
\label{(2.25)}
& \sup_{0 \leq t \leq T} \| \rho_{w tx}(t) \|_{L^{2}(I)} \leq Ch,\\
\label{(2.26)}
& \sup_{0 \leq t \leq T} \| \rho_{w t}(t) \|_{L^{2}(I)} \leq Ch^{2}|\log h|^{2}.
\end{align}
The equations \eqref{(1.13)}, \eqref{(1.14)}, \eqref{(2.19)}--\eqref{(2.21)}, \eqref{(2.23)}--\eqref{(2.26)} together with interpolation and inverse estimates imply that
\begin{align}\label{(2.27)}
\| \uhat \|_{W^{1, \infty}(I)}, \| \uhatt \|_{W^{1, \infty}(I)}, \| \what \|_{W^{1, \infty}(I)}, \| \whatt \|_{W^{1, \infty}(I)} \leq C
\end{align}
uniformly in $h$ and time.

\subsection{Discrete initial data and first estimates}

Let $(u_{h},w_{h}, c_{h})$ be the discrete solution on the time interval $[0,T]$. Define
\begin{align}
& C_{0}:= \sup_{x \in I, \, t \in [0,T]} Q(x,t), & \qquad & C_{1}:= \sup_{x \in [0,1], \, t \in [0,T]} |w(x,t)|, \\ \label{cbounds}
& C_{2}:= \| c \|_{C([0,T], H^{1}(I))}, & \qquad & C_{3}:= \|c \|_{L^{2}((0,T), H^{1}(I))}.
\end{align}
For the discrete solution we observe that on the time interval $[0,\intertime]$ (for $\intertime$ sufficiently small) we have that
\begin{align} \label{(3.2)}
\sup_{x \in I, \, t \in [0,\intertime]} Q_{h}(x,t) &\leq 2 C_{0}, & \qquad \sup_{x \in I, \, t \in [0,\intertime]} |w_{h}(x,t)| &\leq 2C_{1}, \\
 \qquad \label{(3.2)bis} 
\|c_{h} \|_{C([0,\intertime], L^{\infty}(I))} &\leq 2\embconst C_{2}, & \qquad \|c_{h} \|_{L^{2}((0,\intertime), H^{1}(I))} &\leq 2 C_{3}
\end{align}
thanks to the choice of initial conditions (see Lemma~\ref{indata} below), the smoothness assumptions on $(u,c)$, and a continuity argument 
(here, $\embconst$ denotes the constant for the embedding $H^{1}(I) \hookrightarrow L^{\infty}(I)$ which depends on the length of $I$; in our case $I=(0,1)$ one can actually bound $\embconst$ by one). 
Define 
\begin{align} \label{(3.3)}
T_{h} := \sup \{ \intertime \in [0,T]\, | \, (\ref{(3.2)}), (\ref{(3.2)bis}) \text{ hold on } [0,\intertime] \}.
\end{align}
We employ the well known strategy to first derive error estimates on the time interval $[0, T_{h})$ and then use these bounds to infer that $T_{h}=T$. Therefore in what follows we shall assume~\eqref{(3.2)} and \eqref{(3.2)bis} (without specifying this in every statement).
We decompose the errors $u-u_{h}$ and $w-w_{h}$ according to
\begin{align*}
u-u_{h}&= (u-\uhat)+ (\uhat -u_{h})=\rho_{u} + e_{u}, &\qquad& \text{ where } e_{u}:=\uhat -u_{h},\\
w-w_{h} & = (w -\what)+ (\what -w_{h}) =\rho_{w} + e_{w}, &\qquad& \text{ where } e_{w}:=\what -w_{h}.
\end{align*}
Sometimes it is convenient to work with the smooth and discrete unit normals
\begin{align*}
\nu = \frac{(-u_{x},1)}{Q}, \qquad \hat{\nu}_{h} := \frac{(-\uhatx, 1)}{\Qhat}, \qquad \nu_{h}:=\frac{(-u_{hx}, 1)}{Q_{h}}.
\end{align*}
Note that in \cite[(3.4)]{DD06} it is shown that
\begin{align}\label{(3.4)orig}
|\hat{\nu}_{h} -\nu_{h}| \leq |(\uhat- u_{h})_{x}| \leq (1 + \sup_{I} |\uhatx|) Q_{h}|\hat{\nu}_{h} -\nu_{h}|,
\end{align}
which leads to 
\begin{align}\label{(3.4)}
|\hat{\nu}_{h} -\nu_{h}| \leq |(\uhat- u_{h})_{x}|= |e_{ux}| \leq C |\hat{\nu}_{h} -\nu_{h}|,
\end{align}
where the constant $C$ depends on $C_{0}$ and on the constant appearing in \eqref{(2.27)}. Clearly
\begin{align}\label{(3.4)post}
|\Qhat - Q_{h}| \leq \big{|} |(\uhatx,-1)| - |(u_{hx},-1)| \big{|} \leq \big{|} (\uhatx,-1) - (u_{hx},-1) \big{|} = | \uhatx - u_{hx} | = |e_{ux}|.
\end{align}
In the estimates that will follow we will also use the fact that
\begin{align}
\label{pippoQ}
|Q-Q_{h}| &\leq |u_{x}-u_{hx}| \leq |\rho_{ux}| + |e_{ux}|,\\
\label{pipponu}
|\nu -\nu_{h}| &= \left|\frac{(Q_{h}-Q)}{Q_{h}} \frac{1}{Q}(u_{x},-1) + \frac{1}{Q_{h}}(u_{x}-u_{hx}, 0) \right | \leq C |\rho_{ux}| +C |e_{ux}|,
\end{align}
which easily follow employing the boundedness of $Q$ and $Q_{h}$.

We pick the following initial values in \eqref{disicond}: 
\begin{equation} \label{icondh}
 \uich(x) := \hat{\uic}(x), \quad \cich(x) := I_{h}(\cic)(x), \qquad x \in \bar{I},
\end{equation}
where $\hat{\uic}$ is the non-linear projection of $\uic$ defined in \eqref{def-uhat}.
\begin{lemma}\label{indata}
For the choice of initial data in \eqref{icondh} we have that
\begin{align*}
 e_{u}(0) \equiv 0, \qquad \|e_{w}(0)\|_{L^{2}(I) } \leq Ch. 
\end{align*}
\end{lemma}
\begin{proof}
The first statement follows directly from the definition. For the error estimate of $e_{w}(0)$, observe that since $\hat{u}_{h}(\cdot, 0)=u_{0h}(\cdot)$ then by \eqref{(2.11)}, \eqref{def-uhat}, and \eqref{(2.5)} 
\begin{align*}
\int_{I} \frac{w_{h}(0)\xi_{h}}{Q_{h}(0)} = \int_{I} \frac{\uichx \xi_{hx}}{Q_{h}(0)} = \int_{I} \frac{\uicx \xi_{hx}}{Q(0)} = \int_{I} \frac{w(0) \xi_{h}}{Q(0)}
\end{align*}
for any $\xi_{h} \in X_{h0}$. Subtraction gives
\begin{align*}
\int_{I} \left( \frac{w(0)}{Q(0)} - \frac{w_{h}(0)}{Q_{h}(0)} \right) \xi_{h} = 0 \qquad \forall \, \xi_{h} \in X_{h0}.
\end{align*}
Testing with $\xi_{h} = I_{h}(w(0)) - w_{h}(0)$ gives 
\begin{align*}
\int_{I} \frac{|w(0) - w_{h}(0)|^{2}}{Q_{h}(0)} &= \int_{I} w(0)(w(0)-w_{h}(0)) \frac{Q(0)-Q_{h}(0)}{Q(0) Q_{h}(0)}
\\
& \quad + \int_{I} \frac{(w(0)-w_{h}(0))}{Q_{h}(0)}(w(0)-I_{h}(w(0))) \\
& \quad + \int_{I}\frac{w(0)}{Q(0) Q_{h}(0)} (Q_{h}(0)-Q(0))(w(0)-I_{h}(w(0))).
\end{align*}
We infer that $\|w(0)-w_{h}(0)\|_{L^{2}(I)} \leq C h $ by a standard $\epsilon$-Young argument, \eqref{(2.18)}, \eqref{(1.13)}, \eqref{(4.2)}, \eqref{(2.27)}, and the boundedness of $1 \leq Q_{h}(0) \leq C_{0} + C\|\rho_{ux}\|_{L^{\infty}} \leq \frac{3}{2} C_{0}$ by \eqref{(2.19)}, \eqref{(2.27)}, and $h$ small enough. 
The claim now follows by writing $e_{w}(0)=-\rho_{w}(0)+ (w(0)-w_{h}(0))$ and using \eqref{(2.24)}.
\end{proof}

\begin{remark}\label{regularization}
In the discrete formulation of the problem we have introduced a regularization term weigthed by $\mu(h)$. This is motivated by the necessity of finding an error estimate for $|(Q-Q_{h})_{t}|$ in Lemma~\ref{lemma:c-ch} below (cf. term $K_{1}$ in the proof).
Note that we can write
\begin{align*}
|(Q-Q_{h})_{t}| = \left |\frac{u_{hx}}{Q_{h}}(u_{xt}-u_{hxt}) + u_{xt} ( \frac{u_{x}}{Q} - \frac{u_{hx}}{Q_{h}}) \right|
\leq |\rho_{uxt}| + |e_{uxt}| + C |\nu-\nu_{h}|.
\end{align*}
The regularisation helps in deriving an estimate for the ``tricky'' term $|e_{uxt}|$, see \eqref{(3.21)bis} below.
\end{remark}

\subsection{Error estimates for $e_{u}$ and $e_{w}$}

The following error estimates for $e_{u}$ and $e_{w}$ are obtained through appropriate modification of the corresponding error estimates shown in \cite{DD06}. We have used the same notation on purpose so that it will be easier for the reader to look up the details which are not repeated here for the sake of conciseness.
Moreover we give statements in such a way that it is easy to make a distinction as for which contributions come from the ``new'' coupling and regularising terms and those that have a purely geometrical meaning. 

\begin{lemma}\label{lemma3.1}
Suppose that $F:\mbR \to \mbR$ is twice continuously differentiable and that $\zeta \in H^{1}_{0}(I)$. Then
$$ \int_{I}(F(u_{x}) -F(\uhatx)) \zeta \, dx = \int_{I} \rho_{u}\frac{\partial}{\partial x} (\zeta F'(u_{x})) dx +R,$$
where $R$ satisfies $|R| \leq C h^{2}|\log h| \| \zeta \|_{L^{2}(I)}$.
\end{lemma}
\begin{proof}
See \cite[Lemma~3.1]{DD06}. It uses a mean value theorem, the smoothness of $F$ and $u$ (recall \eqref{(1.13)}) and the bounds \eqref{(2.18)}, \eqref{(2.19)}, and \eqref{(2.27)}.
\end{proof}

\begin{lemma}\label{lemma3.2}
For every $\epsilon >0$ there exists $C_{\epsilon}$ such that
\begin{align*}
\| e_{ux}(t) \|_{L^{2}(I)}^{2} \leq \epsilon \| e_{w}(t) \|_{L^{2}(I)}^{2} + C_{\epsilon} \| e_{u}(t) \|_{L^{2}(I)}^{2} + C h^{4} |\log h|^{2}, \qquad 0 \leq t <T_{h}.
\end{align*}
\end{lemma}
\begin{proof}
See \cite[Lemma~3.2]{DD06}. Here one starts from the equation
$$ \int_{I} \left(\frac{\uhatx}{\Qhat} - \frac{u_{hx}}{Q_{h}} \right) \varphi_{hx} \, dx = \int_{I} \left(\frac{w}{Q} - \frac{w_{h}}{Q_{h}}\right) \varphi_{h} \, dx \qquad \forall \, \, \varphi \in X_{h0},$$
which follows from \eqref{def-uhat}, \eqref{(2.5)}, and \eqref{(2.11)}, and tests with $\varphi_{h}=e_{u}$.
\end{proof}
\begin{lemma}\label{lemma3.3}
For $0 \leq t <T_{h}$ we have
\begin{align*}
\| e_{wx}(t) \|_{L^{2}(I)}^{2} &\leq C ( \| e_{ux}(t) \|_{L^{2}(I)}^{2} +\| e_{ut}(t) \|_{L^{2}(I)}^{2} + \| e_{w}(t) \|_{L^{2}(I)}^{2} + h^{4} |\log h|^{4} ) \\
& \qquad + C \| (c-c_{h})(t) \|_{L^{2}(I)}^{2} + C \mu(h)^{2} \|u_{hxt}(t)\|_{L^{2}(I)}^{2}
+ Ch^{2} (1+ \|c_{hx}(t)\|_{L^{2}(I)}^{2}). 
\end{align*}
\end{lemma}
\begin{proof}
The definition \eqref{(2.22)} of $\what$ and \eqref{(2.4)} yield
\begin{align*}
\int_{I} \frac{\uhatt \varphi_{h}}{Q_{h}} &+ \int_{I} E(\uhat) \whatx
 \varphi_{hx} + \frac{1}{2} \int_{I} \frac{\what^{2}}{\Qhat^{3}} \uhatx \varphi_{hx} \\
& =\int_{I} \frac{(\uhatt-u_{t}) \varphi_{h}}{Q_{h}} + \int_{I} u_{t} \left( \frac{1}{Q_{h}} -\frac{1}{Q}\right) \varphi_{h} + \frac{1}{2} \int_{I} (\what^{2}- w^{2})\frac{\uhatx}{\Qhat^{3}} \varphi_{hx} + \int_{I} f(c) \varphi_{h}
 \end{align*}
 for all $\varphi_{h} \in X_{h0}$. Subtracting \eqref{(2.10)} we obtain
 \begin{align}\label{(3.6)}
 \int_{I} \frac{e_{ut} \varphi_{h}}{Q_{h}} & + \int_{I} ( E(\uhatx)\whatx - E(u_{hx})w_{hx}) \varphi_{hx}
 + \frac{1}{2} \left( \frac{\what^{2}}{\Qhat^{3}} \uhatx - \frac{w_{h}^{2}}{Q_{h}^{3}} u_{hx}\right) \varphi_{hx}\\
& =
 -\int_{I} \frac{\rho_{ut} \varphi_{h}}{Q_{h}} + \int_{I} u_{t}\left( \frac{1}{Q_{h}} -\frac{1}{Q}\right) \varphi_{h} + \frac{1}{2} \int_{I} (\what^{2}-w^{2}) \frac{\uhatx}{\Qhat^{3}} \varphi_{hx} \notag \\
 & \qquad + 
 \int_{I} (f(c)-I_{h}(f(c_{h})) )\varphi_{h} + \mu(h) \int_{I} u_{hxt} \varphi_{hx}. \notag
 \end{align}
 After inserting $\varphi_{h}=e_{w} \in X_{h0}$ we derive
 \begin{align}\label{bla}
 \int_{I} ( E(\uhatx)\whatx - E(u_{hx})w_{hx}) e_{wx} &= - \int_{I} \frac{e_{ut} e_{w}}{Q_{h}} 
 -\frac{1}{2} \left( \frac{\what^{2}}{\Qhat^{3}} \uhatx - \frac{w_{h}^{2}}{Q_{h}^{3}} u_{hx}\right) e_{wx}
 -\int_{I} \frac{\rho_{ut} e_{w}}{Q_{h}}\\
 & \quad + \int_{I} u_{t}\left( \frac{1}{Q_{h}} -\frac{1}{Q}\right) e_{w} + \frac{1}{2} \int_{I} (\what^{2}-w^{2}) \frac{\uhatx}{\Qhat^{3}} e_{wx} \notag \\
 & \quad + 
 \int_{I} (f(c)-I_{h}(f(c_{h})) )e_{w} + \mu(h) \int_{I} u_{hxt} e_{wx}. \notag
 \end{align}
 For the last two terms we observe that
 \begin{align*}
 \left |\mu(h) \int_{I} u_{hxt} e_{wx} \right | \leq \epsilon \| e_{wx} \|_{L^{2}(I)}^{2} + C_{\epsilon} 
 \mu(h)^{2} \| u_{htx} \|_{L^{2}(I)}^{2},
 \end{align*} 
 and
 \begin{align}
 \left|\int_{I} (f(c)-I_{h}(f(c_{h})) )e_{w} \right| &\leq \left|\int_{I} (f(c)-f(c_{h}) )e_{w}\right| + \left |\int_{I} (f(c_{h})-I_{h}(f(c_{h})) )e_{w} \right | \nonumber \\
 & \leq C \|c-c_{h} \|_{L^{2}(I)}^{2} + Ch^{2} (1+ \int_{I} |c_{hx}|^{2}) + C\| e_{w} \|_{L^{2}(I)}^{2} \label{ieq:f}
 \end{align}
 where we have used \eqref{fbound} and \eqref{(4.2)}.
 From now on we argue exactly as in \cite[Lemma~3.3]{DD06}.
 The error bound relies on the fact that it can be shown that
 $$ \int_{I} ( E(\uhatx)\whatx - E(u_{hx})w_{hx}) e_{wx} \geq \frac{1}{2\sqrt{1+ 4C_{0}^{2}}} \| e_{wx} \|_{L^{2}(I)}^{2} -C \| e_{ux} \|_{L^{2}(I)}^{2}.$$
 The estimates for the remaining terms on the right-handside of \eqref{bla} are carefully explained in \cite[Lemma~3.3]{DD06}, hence we do not repeat the arguments here. 
\end{proof}

\begin{lemma} \label{lemma3.4}
For $0 \leq t < T_{h}$ we have
\begin{align*}
&\frac{\mu(h)}{2} \|e_{utx} \|_{L^{2}(I)}^{2} + \frac{1}{4C_{0}} \|e_{ut} \|_{L^{2}(I)}^{2}
+\int_{I} ( E(\uhatx)\whatx - E(u_{hx})w_{hx}) e_{utx} + \frac{1}{2}\int_{I} \left( \frac{\what^{2}}{\Qhat^{3}} \uhatx - \frac{w_{h}^{2}}{Q_{h}^{3}} u_{hx}\right) e_{utx} \\
& \leq -\frac{d}{dt} \int_{I} u_{t} \frac{u_{x}}{Q^{3}} e_{ux} \rho_{u} +\frac{1}{2} \frac{d}{dt} \int_{I} (\what^{2}-w^{2}) \frac{\uhatx}{\Qhat^{3}} e_{ux} + C \| e_{ux} \|_{L^{2}(I)}^{2} + Ch^{4} |\log h|^{4} \\
& \qquad +
C \|c-c_{h} \|_{L^{2}(I)}^{2} + C \mu(h)^{2} + C \mu(h)h^{2} +Ch^{2}(1 + \|c_{hx}\|_{L^{2}(I)}^{2}).
\end{align*} 
\end{lemma}

\begin{proof}
Choosing $\varphi=e_{ut}\in X_{h0}$ in \eqref{(3.6)} and using \eqref{(3.2)} we obtain
\begin{align*}
\int_{I} \frac{e_{ut}^{2}}{2C_{0}} & + \int_{I} ( E(\uhatx)\whatx - E(u_{hx})w_{hx}) e_{utx}
 + \frac{1}{2} \left( \frac{\what^{2}}{\Qhat^{3}} \uhatx - \frac{w_{h}^{2}}{Q_{h}^{3}} u_{hx}\right) e_{utx}\\
& \leq
 -\int_{I} \frac{\rho_{ut} e_{ut}}{Q_{h}} + \int_{I} u_{t}\left( \frac{1}{Q_{h}} -\frac{1}{\Qhat}\right) e_{ut} + 
 \int_{I} u_{t}\left( \frac{1}{\Qhat} -\frac{1}{Q}\right) e_{ut}
+ \frac{1}{2} \int_{I} (\what^{2}-w^{2}) \frac{\uhatx}{\Qhat^{3}} \varphi_{hx} \\
 & \qquad + 
 \int_{I} (f(c)-I_{h}(f(c_{h})) )e_{ut} + \mu(h) \int_{I} u_{hxt} e_{utx} =: I +II+III+IV+V+VI. 
\end{align*}
The terms $I,II,III, IV$ are treated and estimated as in \cite[Lemma~3.4]{DD06}. Again we refrain from giving details here since the original paper gives all argument in detail. For the fifth term we proceed as in \eqref{ieq:f} but with an $\epsilon$ weight and obtain that
\[
 \left|\int_{I} (f(c)-I_{h}(f(c_{h})) )e_{ut} \right| \leq C_{\epsilon} \|c-c_{h} \|_{L^{2}(I)}^{2} + 2\epsilon \| e_{ut} \|_{L^{2}(I)}^{2} + C_{\epsilon}h^{2} (1+ \int_{I} |c_{hx}|^{2}).
\]
For the last term we compute using integration by parts (recall that $e_{ut}=0$ on $\pd I$)
\begin{align*}
VI &= \mu(h) \int_{I} u_{htx} e_{utx}= \mu(h) \left( \int_{I} (u_{htx} -\uhattx) e_{utx} + \int_{I} (\uhattx -u_{tx}) e_{utx} + \int_{I} u_{tx} e_{utx}\right)\\
&= -\mu(h) \|e_{utx}\|_{L^{2}(I)}^{2} - \mu(h)\int_{I} \rho_{utx} e_{utx} - \mu(h)\int_{I} u_{txx} e_{ut}\\
& \leq -\mu(h) \|e_{utx}\|_{L^{2}(I)}^{2} + \frac{\mu(h)}{2} \|e_{utx}\|_{L^{2}(I)}^{2} + \frac{\mu(h)}{2} Ch^{2} + \epsilon \|e_{ut}\|_{L^{2}(I)}^{2} + C_{\epsilon} \mu(h)^{2},
\end{align*}
where we have used \eqref{(2.21)} and \eqref{(1.14)}. An appropriate choice of $\epsilon$ together with the estimates for the terms $I$--$VI$ gives the claim.
\end{proof}

\begin{lemma}\label{lemma3.5}
For $0 \leq t < T_{h}$ we have
\begin{multline*}
\frac{1}{2} \frac{d}{dt} \int_{I} \frac{e_{w}^{2}}{Q_{h}} -\frac{1}{2} \int_{I} \frac{e_{w}^{2}}{Q_{h}^{2}}Q_{ht}
-\int_{I} \what \left( \frac{\Qhatt}{\Qhat^{2}} - \frac{Q_{ht}}{Q_{h}^{2}}\right) e_{w}
-\int_{I} ( E(\uhatx)\uhattx - E(u_{hx})u_{htx}) e_{wx}\\
\leq \epsilon \|e_{wx}\|_{L^{2}(I)}^{2} + C_{\epsilon} ( \|e_{ux}\|_{L^{2}(I)}^{2}+ \|e_{w}\|_{L^{2}(I)}^{2})
+ C_{\epsilon} h^{4}|\log h|^{4}.
\end{multline*}
\end{lemma}
\begin{proof} See \cite[Lemma~3.5]{DD06}: the main idea is to take
equation \eqref{(2.5)} together with the definition of $\uhat$ (recall \eqref{def-uhat}) to infer
$$ \int_{I} \frac{w \xi_{h}}{Q}= \int_{I} \frac{u_{x} \xi_{hx}}{Q}= \int_{I} \frac{\uhatx \xi_{hx}}{\Qhat} \,\qquad \forall \xi_{h } \in X_{h0},$$
and from which differentiation with respect to time gives
\begin{align}\label{eq-uno}
 \int_{I} \frac{w_{t} \xi_{h}}{Q} - \int_{I} \frac{w \xi_{h}}{Q^{2}}Q_{t} - \int_{I} E(\uhatx) \uhattx \xi_{hx}=0 \,\qquad \forall \xi_{h } \in X_{h0}.
 \end{align}
Differentiation with respect to time of \eqref{(2.11)} gives 
\begin{align}\label{eq-due}
 \int_{I} \frac{w_{ht} \xi_{h}}{Q_{h}} - \int_{I} \frac{w_{h} \xi_{h}}{Q_{h}^{2}}Q_{ht} - \int_{I} E(u_{hx}) u_{htx} \xi_{hx}=0 \,\qquad \forall \xi_{h } \in X_{h0}.
 \end{align}
The claim now follows by taking the difference of \eqref{eq-uno}, \eqref{eq-due}, and testing with $\xi_{h}=e_{w}$.
\end{proof}

It follows now from Lemma~\ref{lemma3.4} and Lemma~\ref{lemma3.5} that
\begin{align}
\frac{\mu(h)}{2} & \|e_{utx} \|_{L^{2}(I)}^{2} + \frac{1}{4C_{0}} \|e_{ut} \|_{L^{2}(I)}^{2} + \frac{1}{2} \frac{d}{dt} \int_{I} \frac{e_{w}^{2}}{Q_{h}} \notag \\
& + \frac{1}{2}\int_{I} \left( \frac{\what^{2}}{\Qhat^{3}} \uhatx - \frac{w_{h}^{2}}{Q_{h}^{3}} u_{hx}\right) e_{utx}
-\frac{1}{2} \int_{I} \frac{e_{w}^{2}}{Q_{h}^{2}}Q_{ht}
-\int_{I} \what \left( \frac{\Qhatt}{\Qhat^{2}} - \frac{Q_{ht}}{Q_{h}^{2}}\right) e_{w} \notag \\
& + \int_{I} ( E(\uhatx)\whatx - E(u_{hx})w_{hx}) e_{utx} -\int_{I} ( E(\uhatx)\uhattx - E(u_{hx})u_{htx}) e_{wx}
 \notag \\
\leq & -\frac{d}{dt} \int_{I} u_{t} \frac{u_{x}}{Q^{3}} e_{ux} \rho_{u} +\frac{1}{2} \frac{d}{dt} \int_{I} (\what^{2}-w^{2}) \frac{\uhatx}{\Qhat^{3}} e_{ux} + C_{\epsilon }h^{4} |\log h|^{4} \notag \\
& + \epsilon \|e_{wx}\|_{L^{2}(I)}^{2} + C_{\epsilon} ( \|e_{ux}\|_{L^{2}(I)}^{2}+ \|e_{w}\|_{L^{2}(I)}^{2}) \notag \\
& + C \|c-c_{h} \|_{L^{2}(I)}^{2} + C \mu(h)^{2} + C \mu(h)h^{2} +Ch^{2}(1 + \|c_{hx}\|_{L^{2}(I)}^{2}). \label{(3.13)}
\end{align}
The terms appearing is the second and third line are dealt with as in \cite[p34--37]{DD06} where the lengthy calculations are presented in detail. We thus only list the relevant results. Precisely one finds that (see \cite[(3.15), (3.18), (3.19)]{DD06})
\begin{align}\label{(3.15)}
&\frac{1}{2}\int_{I} \left( \frac{\what^{2}}{\Qhat^{3}} \uhatx - \frac{w_{h}^{2}}{Q_{h}^{3}} u_{hx}\right) e_{utx}
-\frac{1}{2} \int_{I} \frac{e_{w}^{2}}{Q_{h}^{2}}Q_{ht}
-\int_{I} \what \left( \frac{\Qhatt}{\Qhat^{2}} - \frac{Q_{ht}}{Q_{h}^{2}}\right) e_{w} \\
& \qquad \geq \frac{1}{2} \frac{d}{dt} \int_{I} \what^{2} \left \{ \frac{1}{2} \frac{Q_{h}}{\Qhat^{2}} |\hat{\nu}_{h} -\nu_{h}|^{2} - \frac{1}{Q_{h}\Qhat^{2}}(\Qhat -Q_{h})^{2}\right \} - C (\| e_{w} \|_{L^{2}(I)}^{2} + \| e_{ux} \|_{L^{2}(I)}^{2} ),\notag\end{align}
as well as
\begin{align}\label{(3.18)}
&\int_{I} (E(\uhatx) -E(u_{hx}))(\uhattx -u_{htx}) \whatx \\
& \qquad \geq \frac{d}{dt} \int_{I} \left( \left(\frac{Q_{h}}{\Qhat} -1\right)(\hat{\nu}_{h} -\nu_{h})
-\frac{1}{2} \frac{Q_{h}}{\Qhat} |\hat{\nu}_{h} -\nu_{h}|^{2} \hat{\nu}_{h}\right) \cdot (\whatx, 0)^{t} - C \|e_{ux}\|_{L^{2}(I)}^{2} , \notag
\end{align}
and
\begin{align}\label{(3.19)}
 \left| \int_{I} (E(\uhatx)-E(u_{hx}))\uhattx e_{wx}\right| \leq \epsilon \|e_{wx}\|_{L^{2}(I)}^{2} + C_{\epsilon} \|e_{ux}\|_{L^{2}(I)}^{2}. 
\end{align}
Observing that
\begin{multline*}
 (E(\uhatx)\whatx - E(u_{hx})w_{hx}) e_{utx} - ( E(\uhatx)\uhattx - E(u_{hx})u_{htx}) e_{wx}\\=
(E(\uhatx) -E(u_{hx}))(\uhattx -u_{htx}) \whatx - (E(\uhatx)-E(u_{hx}))\uhattx e_{wx},
\end{multline*}
and inserting \eqref{(3.15)}, \eqref{(3.18)}, \eqref{(3.19)} into \eqref{(3.13)} we obtain
\begin{align*}
\frac{\mu(h)}{2} & \|e_{utx} \|_{L^{2}(I)}^{2} + \frac{1}{4C_{0}} \|e_{ut} \|_{L^{2}(I)}^{2} + \frac{1}{2} \frac{d}{dt} \int_{I} \frac{e_{w}^{2}}{Q_{h}} \\
\leq & -\frac{d}{dt} \int_{I} u_{t} \frac{u_{x}}{Q^{3}} e_{ux} \rho_{u} +\frac{1}{2} \frac{d}{dt} \int_{I} (\what^{2}-w^{2}) \frac{\uhatx}{\Qhat^{3}} e_{ux} \\
& -\frac{1}{2} \frac{d}{dt} \int_{I} \what^{2} \left \{ \frac{1}{2} \frac{Q_{h}}{\Qhat^{2}} |\hat{\nu}_{h} -\nu_{h}|^{2} - \frac{1}{Q_{h}\Qhat^{2}}(\Qhat -Q_{h})^{2}\right \}\\
& - \frac{d}{dt} \int_{I} \left( \left(\frac{Q_{h}}{\Qhat} -1\right)(\hat{\nu}_{h} -\nu_{h})
-\frac{1}{2} \frac{Q_{h}}{\Qhat} |\hat{\nu}_{h} -\nu_{h}|^{2} \hat{\nu}_{h}\right) \cdot (\whatx, 0)^{t} 
 \\
& + C_{\epsilon }h^{4} |\log h|^{4} 
 +\epsilon \|e_{wx}\|_{L^{2}(I)}^{2} 
+ C_{\epsilon} ( \|e_{ux}\|_{L^{2}(I)}^{2}+ \|e_{w}\|_{L^{2}(I)}^{2}) \notag \\
& + C \|c-c_{h} \|_{L^{2}(I)}^{2} + C \mu(h)^{2} + C \mu(h)h^{2} +Ch^{2}(1 + \|c_{hx}\|_{L^{2}(I)}^{2}). 
\end{align*}
Integration with respect to time for some $\bar{t} \in (0,T_{h})$, application of Lemma~\ref{lemma3.2} and Lemma~\ref{lemma3.3}, \eqref{(2.27)}, \eqref{(3.4)}, \eqref{(3.2)}, \eqref{(3.2)bis}, \eqref{(2.18)}, and using the approximation order of the initial data (recall Lemma~\ref{indata}) yields
\begin{align}
\int_{0}^{\bar{t}} \mu(h) & \|e_{utx} \|_{L^{2}(I)}^{2} dt + \int_{0}^{\bar{t}} \|e_{ut} \|_{L^{2}(I)}^{2} dt + \| e_{w}(\bar{t}) \|_{L^{2}(I)}^{2} \nonumber \\
& \leq C \| e_{w}(0) \|_{L^{2}(I)}^{2}+C \| e_{ux}(\bar{t}) \|_{L^{2}(I)} ( \| \rho_{u}(\bar{t}) \|_{L^{2}(I)} + \| e_{w}(\bar{t}) \|_{L^{2}(I)} + \| e_{ux}(\bar{t}) \|_{L^{2}(I)} ) \nonumber \\
& \qquad + C \| e_{ux}(0) \|_{L^{2}(I)} ( \| \rho_{u}(0) \|_{L^{2}(I)} + \| e_{w}(0) \|_{L^{2}(I)} + \| e_{ux}(0) \|_{L^{2}(I)} )  \displaybreak[0] \nonumber \\
& \qquad + C_{\epsilon }h^{4} |\log h|^{4} +\epsilon \int_{0}^{\bar{t}} \|e_{wx}\|_{L^{2}(I)}^{2} dt \notag \displaybreak[0] \\
& \qquad + C_{\epsilon} \int_{0}^{\bar{t}} ( \|e_{ux}\|_{L^{2}(I)}^{2}+ \|e_{w}\|_{L^{2}(I)}^{2}) dt + C \int_{0}^{\bar{t}} \|c-c_{h} \|_{L^{2}(I)}^{2} dt \nonumber \\
& \qquad + C \mu(h)^{2} + C \mu(h)h^{2} +Ch^{2}(1 + \int_{0}^{\bar{t}} \|c_{hx}\|_{L^{2}(I)}^{2} dt ) \displaybreak[0] \nonumber \\
& \leq \epsilon \left ( \|e_{w}(\bar{t}) \|_{L^{2}(I)}^{2} +\int_{0}^{\bar{t}} \|e_{ut} \|_{L^{2}(I)}^{2} dt \right) + \epsilon C \mu(h)^{2} \int_{0}^{\bar{t}} \| u_{hxt}\|_{L^{2}(I)}^{2} dt \displaybreak[0] \nonumber \\
& \qquad + C_{\epsilon} \left( \|e_{u}(\bar{t}) \|_{L^{2}(I)}^{2} + h^{4}|\log h|^{4} + \int_{0}^{\bar{t}} 
(\|e_{u} \|_{L^{2}(I)}^{2}+ \|e_{w} \|_{L^{2}(I)}^{2}) dt \right) \nonumber \\
&\qquad + C \int_{0}^{\bar{t}} \|c-c_{h} \|_{L^{2}(I)}^{2} dt + C \mu(h)^{2} + C \mu(h)h^{2} +Ch^{2}. \label{ieq:hilf}
\end{align}
Thanks to \eqref{(2.27)} and as $\mu(h) \leq 1$ for all $h \leq h_{0}$ with some sufficiently small $h_{0}$ we have that
\begin{multline}\label{stern}
\epsilon C\mu(h)^{2} \int_{0}^{\bar{t}} \| u_{hxt}\|_{L^{2}(I)}^{2} dt \\
\leq \epsilon C\mu(h)^{2} \int_{0}^{\bar{t}} (\| e_{utx}\|_{L^{2}(I)}^{2} + C ) dt \leq \epsilon C \mu(h) \int_{0}^{\bar{t}} \| e_{utx}\|_{L^{2}(I)}^{2} dt + \epsilon C \mu(h)^{2}.
\end{multline}
Moreover, using that $e_{u}(0)=0$ we obtain that
\[ 
\|e_{u}(\bar{t}) \|_{L^{2}(I)}^{2} = \int_{0}^{\bar{t}} \frac{d}{dt} \big{\|} e_{u}(t) \big{\|}_{L^{2}(I)}^{2} dt = \int_{0}^{\bar{t}} \int_I 2 e_{u} e_{ut} dx dt \leq \epsilon \int_{0}^{\bar{t}} \|e_{ut} \|_{L^{2}(I)}^{2} dt + C_{\epsilon} \int_{0}^{\bar{t}} \|e_{u} \|_{L^{2}(I)}^{2} dt.
\] 
Using this and \eqref{stern} with $\epsilon$ small enough in \eqref{ieq:hilf} yields
\begin{multline}
\int_{0}^{\bar{t}} \mu(h) \|e_{utx} \|_{L^{2}(I)}^{2} dt + \int_{0}^{\bar{t}} \|e_{ut} \|_{L^{2}(I)}^{2} dt + \| e_{w}(\bar{t}) \|_{L^{2}(I)}^{2} + \|e_{u}(\bar{t}) \|_{L^{2}(I)}^{2} \\
\leq C\int_{0}^{\bar{t}} (\|e_{u} \|_{L^{2}(I)}^{2}+ \|e_{w} \|_{L^{2}(I)}^{2}) dt + C \int_{0}^{\bar{t}} \|c-c_{h} \|_{L^{2}(I)}^{2} dt + C \mu(h)^{2} + C \mu(h)h^{2} +Ch^{2}. \label{ieq:mu_critical}
\end{multline}
A Gronwall argument and using that $\mu(h) \leq C h$ for all $h \leq h_{0}$ (after eventually reducing $h_{0}$) finally yields
\begin{align} \label{(3.21)}
\| e_{w}(\bar{t}) \|_{L^{2}(I)}^{2} + \|e_{u}(\bar{t}) \|_{L^{2}(I)}^{2} \leq C \int_{0}^{\bar{t}} \|c-c_{h} \|_{L^{2}(I)}^{2} dt +Ch^{2}, \qquad \bar{t}\in [0, T_{h}),
\end{align}
from which also conclude that
\begin{align}\label{(3.21)bis}
 \mu(h) \int_{0}^{\bar{t}} \|e_{utx} \|_{L^{2}(I)}^{2} dt + \int_{0}^{\bar{t}} \|e_{ut} \|_{L^{2}(I)}^{2} dt \leq
C \int_{0}^{\bar{t}} \|c-c_{h} \|_{L^{2}(I)}^{2} dt +Ch^{2}, \qquad \bar{t}\in [0, T_{h}).
\end{align}

\subsection{Error estimate for $(c-c_{h})$}

In order to proceed we need to analyse the error between $c$ and $c_{h}$. We here basically follow the lines of \cite{PS17}, Lemma 4.2. But we need to provide all details as the treatment of the terms with the time derivative of the length element is different here.

\begin{lemma}\label{lemma:c-ch}
We have that for any ${\intertime} \in [0, T_{h})$
\begin{align*}
\| c({\intertime}) - & c_{h}({\intertime}) \|_{L^{2}(I)}^{2} + \int_0^{\intertime} \| c_{x} - c_{hx} \|_{L^{2}(I)}^{2} dt \nonumber \\
& \leq C \|e_{ux} ({\intertime})\|_{L^{2}(I)}^{2} + C \int_{0}^{{\intertime}} \| c - c_{h} \|_{L^{2}(I)}^{2} dt 
+ C \int_{0}^{{\intertime}} \|e_{ux}\|_{L^{2}(I)}^{2} dt + C \int_{0}^{{\intertime}} \|e_{ut}\|_{L^{2}(I)}^{2} dt \\
& \quad + C \int_{0}^{{\intertime}} \|e_{ux}\|_{L^{2}(I)}^{2} \|e_{utx}\|_{L^{2}(I) }^{2} dt
+ C h^{2} \int_{0}^{{\intertime}} \|e_{utx}\|_{L^{2}(I)}^{2} dt + C h^{2}.
\end{align*}
\end{lemma}
\begin{proof}
The difference between the continuous \eqref{c-gleichung-long} and the discrete version \eqref{ch-gleichung-long} reads
\[
\int_{I} (c Q - c_{h} Q_{h})_{t} \zeta_{h} \, dx + \int_{I} \Big{(} \frac{c_{x}}{Q} - \frac{c_{hx}}{Q_{h}} \Big{)} \zeta_{hx} \, dx = 0
\]
for all test functions $\zeta_{h}(x,t)$ of the form $\zeta_{h} = \sum_{j=1}^{N-1} \zeta_j(t) \varphi_j(x)$. Choosing 
\[
 \zeta_{h} = I_{h}(c) - c_{h} = c - c_{h} + I_{h}(c) - c
\]
a calculation (cf. \cite[Lemma~4.2]{PS17}) yields that
\begin{align}
\frac{d}{dt} \Big{(} \int_{I} & \frac{1}{2} (c - c_{h})^{2} Q_{h} \, dx \Big{)} + \int_{I} \frac{|(c - c_{h})_{x}|^{2}}{Q_{h}} \, dx \nonumber \\
&= \int_{I} \big{(} c (Q_{h} - Q) \big{)}_{t} (c - c_{h}) \, dx - \int_{I} \frac{1}{2} (c - c_{h})^{2} Q_{ht} \, dx \nonumber \displaybreak[0] \\
&+ \frac{d}{dt} \Big{(} \int_{I} (c Q - c_{h} Q_{h}) (c - I_{h}(c)) \, dx \Big{)} \displaybreak[0] \nonumber \\
&- \int_{I} (c Q - c_{h} Q_{h}) \big{(} c - I_{h}(c) \big{)}_{t} \, dx \displaybreak[0] \nonumber \\
&+ \int_{I} \frac{(c - c_{h})_{x} (c - I_{h}(c))_{x}}{Q_{h}} \, dx \displaybreak[0] \nonumber \\
&+ \int_{I} c_{x} \frac{(c - c_{h})_{x}}{\sqrt{Q_{h}}} \frac{Q - Q_{h}}{\sqrt{Q_{h}} Q} \, dx + \int_{I} c_{x} (I_{h}(c) - c)_{x} \frac{Q - Q_{h}}{Q_{h} \, Q} \, dx \nonumber \\
&= \sum_{j=1}^{7} K_j \label{eq:toest_c}.
\end{align}

For the first term we can write
\begin{align*}
-K_1 &= \int_{I} c_{t} (Q-Q_{h}) (c - c_{h}) \, dx + \int_{I} c ( Q-\Qhat)_{t} (c - c_{h}) \, dx + \int_{I} c ( \Qhat-Q_{h})_{t} (c - c_{h}) \, dx\\
&=:K_{1,0}+K_{1,1}+ K_{1,2}.
\end{align*}
Using \eqref{pippoQ}, the smoothness assumptions on $c$ (recall \eqref{cond-c}) and \eqref{(2.18)} we infer immediately that
\begin{align*}
|K_{1,0}| \leq C \|c-c_{h}\|_{L^{2}(I)} (\|\rho_{ux}\|_{L^{2}(I)}+ \|e_{ux}\|_{L^{2}(I)}) \leq C \|c-c_{h}\|_{L^{2}(I)}^{2} + C\|e_{ux}\|_{L^{2}(I)}^{2} + C h^{2}.
\end{align*}
Next we write using \eqref{(2.21)}, \eqref{(2.27)}, the fact that $|\nu-\hat{\nu}_{h}| \leq C |\rho_{ux}|$ and \eqref{(2.18)}
\begin{align*}
|K_{1,1}| & = \left|\int_{I} c ( Q-\Qhat)_{t} (c - c_{h}) \right| \\
& = \left| \int_{I} c (c - c_{h})(\frac{u_{x}}{Q} - \frac{\uhatx}{\Qhat})\uhattx + \int_{I} c (c - c_{h})\frac{u_{x}}{Q}(u_{tx}-\uhattx) \right|
\\
& \leq C\|c-c_{h}\|_{L^{2}(I)} \|\nu- \hat{\nu}_{h}\|_{L^{2}(I)} + C \|c-c_{h}\|_{L^{2}(I)} \|\rho_{utx}\|_{L^{2}(I)}\\
& \leq Ch\|c-c_{h}\|_{L^{2}(I)}\leq C\|c-c_{h}\|_{L^{2}(I)}^{2} + Ch^{2}.
\end{align*}
For the last term we observe using partial integration that
\begin{align*}
K_{1,2}& =\int_{I} c ( \Qhat-Q_{h})_{t} (c - c_{h}) = \int_{I} c (c - c_{h}) \left(\frac{\uhatx}{\Qhat} \uhattx -\frac{u_{hx}}{Q_{h}}u_{htx} \right)\\
& =\int_{I} c (c - c_{h}) \uhattx \left(\frac{\uhatx}{\Qhat} -\frac{u_{hx}}{Q_{h}}\right) 
+ \int_{I} c (c - c_{h}) \left(\frac{u_{hx}}{Q_{h}} -\frac{u_{x}}{Q} \right) (\uhattx -u_{htx})\\
& \quad +\int_{I} c (c - c_{h})\frac{u_{x}}{Q} (\uhattx -u_{htx}) \\
& =\int_{I} c (c - c_{h}) \uhattx \left(\frac{\uhatx}{\Qhat} -\frac{u_{hx}}{Q_{h}}\right) 
+ \int_{I} c (c - c_{h}) \left(\frac{u_{hx}}{Q_{h}} -\frac{u_{x}}{Q} \right) (\uhattx -u_{htx})\\
& \quad - \int_{I} \frac{\partial}{\partial x}\left( c (c - c_{h})\frac{u_{x}}{Q}\right) (\uhatt -u_{ht}) .
\end{align*}
Therefore we infer using \eqref{(2.27)}, \eqref{(3.4)}, \eqref{pipponu}, \eqref{(2.18)}, and embedding theory that
\begin{align*}
|K_{1,2}| & \leq C\|c-c_{h}\|_{L^{2}(I)} \|e_{ux}\|_{L^{2}(I)} + C \|e_{ut}\|_{L^{2}(I)} ( \|c-c_{h}\|_{L^{2}(I)} +
\|(c-c_{h})_{x}\|_{L^{2}(I)}) \\
& \quad+ C \|c-c_{h}\|_{L^{\infty}(I)} \|\nu-\nu_{h}\|_{L^{2}(I)} \|e_{utx}\|_{L^{2}(I)} \\
& \leq C \|c-c_{h}\|_{L^{2}(I)}^{2} + C \|e_{ux}\|_{L^{2}(I)}^{2} + \epsilon \|(c-c_{h})_{x}\|_{L^{2}(I)}^{2} + C_{\epsilon} \|e_{ut}\|_{L^{2}(I)}^{2} \\
& \quad + C \|c-c_{h}\|_{H^{1}(I)} (h+ \|e_{ux}\|_{L^{2}(I)})\|e_{utx}\|_{L^{2}(I)} \\
& \leq C \|c-c_{h}\|_{L^{2}(I)}^{2} + C \|e_{ux}\|_{L^{2}(I)}^{2} + \epsilon \|(c-c_{h})_{x}\|_{L^{2}(I)}^{2} + C_{\epsilon} \|e_{ut}\|_{L^{2}(I)}^{2} \\
& \quad + C_{\epsilon} \|e_{ux}\|_{L^{2}(I)}^{2} \|e_{utx}\|_{L^{2}(I) }^{2}
+ C_{\epsilon} h^{2} \|e_{utx}\|_{L^{2}(I)}^{2}.
\end{align*}
Putting all previous estimate together we infer that
\begin{align}
|K_{1}| \leq & C \|c-c_{h}\|_{L^{2}(I)}^{2} + C \|e_{ux}\|_{L^{2}(I)}^{2} + \epsilon \|(c-c_{h})_{x}\|_{L^{2}(I)}^{2}\\
& \quad + C_{\epsilon} \|e_{ut}\|_{L^{2}(I)}^{2} + C_{\epsilon} \|e_{ux}\|_{L^{2}(I)}^{2} \|e_{utx}\|_{L^{2}(I) }^{2}
+ C_{\epsilon} h^{2} \|e_{utx}\|_{L^{2}(I)}^{2} + C h^{2}. \notag
\end{align}

The term $K_{2}$ can be estimated as follows using integration by parts, the fact that $\| c-c_{h}\|_{L^{\infty}} $ is bounded (thanks to \eqref{(3.2)bis}), \eqref{(2.27)}, and \eqref{(2.20)}:
\begin{align*}
|K_{2}|&=\left |\frac{1}{2} \int_{I} (c-c_{h})^{2} Q_{ht}\right| \leq \left|\frac{1}{2} \int_{I} (c-c_{h})^{2} Q_{t} \right| + \left|\frac{1}{2} \int_{I} (c-c_{h})^{2} (Q_{ht} -Q_{t})\right| \\
& \leq C \|c-c_{h}\|_{L^{2}(I)}^{2} + \left|\frac{1}{2} \int_{I} (c-c_{h})^{2} \Big{[} ( \frac{u_{hx}}{Q_{h}}-\frac{u_{x}}{Q} ) u_{htx} +\frac{u_{x}}{Q}(u_{htx} -u_{tx}) \Big{]} \right| \\
& \leq C \|c-c_{h}\|_{L^{2}(I)}^{2} + \|c-c_{h}\|_{L^{\infty}(I)} \left|\frac{1}{2} \int_{I} (c-c_{h}) ( \frac{u_{hx}}{Q_{h}}-\frac{u_{x}}{Q} ) (u_{htx} -\uhattx ) \right| \\
& \quad+ \|c-c_{h}\|_{L^{\infty}(I)} \left|\frac{1}{2} \int_{I} (c-c_{h}) ( \frac{u_{hx}}{Q_{h}}-\frac{u_{x}}{Q} ) \uhattx \right| +\left|\frac{1}{2} \int_{I} \frac{\partial}{\partial x} \left((c-c_{h})^{2} \frac{u_{x}}{Q} \right)(u_{ht} -u_{t}) \right|\\
& \leq C \|c-c_{h}\|_{L^{2}(I)}^{2} + C \|c-c_{h}\|_{L^{\infty}(I)} \|\nu_{h} -\nu\|_{L^{2}(I)} \|e_{utx}\|_{L^{2}(I)} + C \|\nu_{h}-\nu \|_{L^{2}(I)}^{2} \\
& \qquad + \epsilon \|(c-c_{h})_{x}\|_{L^{2}(I)}^{2} + C_{\epsilon} \|e_{ut}\|_{L^{2}(I)}^{2} + C_{\epsilon} h^{4}|\log h|^{4}.
\end{align*}
The second term in the last line of above inequality can be treated as the same term appearing in $K_{1,2}$, so that we obtain
\begin{align}
|K_{2}| \leq & C \|c-c_{h}\|_{L^{2}(I)}^{2} + C \|e_{ux}\|_{L^{2}(I)}^{2} + \epsilon \|(c-c_{h})_{x}\|_{L^{2}(I)}^{2}\\
& \quad + C_{\epsilon} \|e_{ut}\|_{L^{2}(I)}^{2} + C_{\epsilon} \|e_{ux}\|_{L^{2}(I)}^{2} \|e_{utx}\|_{L^{2}(I) }^{2}
+ C_{\epsilon} h^{2} \|e_{utx}\|_{L^{2}(I)}^{2} + C_{\epsilon} h^{2}. \notag
\end{align}

The remaining terms $K_{3}, \ldots, K_{7}$ are estimated as in \cite[Lemma~4.2]{PS17}. Precisely: 
For $K_{3}$ we note that by \eqref{cond-c}, \eqref{(3.2)}, \eqref{(4.2)}, \eqref{pippoQ}, and \eqref{(2.18)}
\begin{align}
& \Big |\int_{I} (c Q - c_{h} Q_{h}) (c - I_{h}(c)) \, dx \Big | \nonumber \\
& = \Big | \int_{I} (c - c_{h}) Q_{h} (c - I_{h}(c)) \, dx + \int_{I} c (Q - Q_{h}) (c - I_{h}(c)) \, dx \Big | \nonumber \\
& \leq \hat{\eps} \int_{I} (c - c_{h})^{2} Q_{h} \, dx + C \int_{I} (Q - Q_{h})^{2} \, dx + C_{\hat{\eps}} h^4 \| c \|_{H^{2}(I)}^{2} \nonumber\\
& \leq \hat{\eps}\int_{I} (c - c_{h})^{2} Q_{h} \, dx + C \| e_{ux} \|_{L^{2}(I)}^{2} + C_{\hat{\eps}} h^{2}
\label{eq:estK3}
\end{align}
with $\hat{\eps}>0$ that will be picked later on. We will refer to this estimate later on when integrating \eqref{eq:toest_c} with respect to time. 

For the term $K_{4}$ we infer from \eqref{(4.2)}, \eqref{(3.2)}, \eqref{cond-c}, \eqref{pippoQ}, and \eqref{(2.18)}, that
\begin{align*}
|K_4| &= \Big{|} \int_{I} c (Q - Q_{h}) (c_{t} - I_{h}(c_{t})) \, dx + \int_{I} (c - c_{h}) (c_{t} - I_{h}(c_{t})) Q_{h} \, dx \Big{|} \\
&\leq C \int_{I} (Q - Q_{h})^{2} \, dx + C \int_{I} (c - c_{h})^{2} Q_{h} \, dx + C \|c_{t}\|_{H^1(I)}^{2} h^{2}\\
& \leq C \|c-c_{h}\|_{L^{2}(I)}^{2} +C \| e_{ux} \|_{L^{2}(I)}^{2} + C h^{2}.
\end{align*}
By the interpolation estimates \eqref{(4.2)}, \eqref{(4.2)bis}, embedding theory, \eqref{cond-c}, \eqref{pippoQ}, and \eqref{(2.18)} we have the following estimates for the terms involving spatial gradients (for $\epsilon > 0$ arbitrarily small):
\begin{align*}
|K_5| &\leq \epsilon \int_{I} \frac{|(c - c_{h})_{x}|^{2}}{Q_{h}} \, dx + C_{\epsilon} \int_{I} \frac{|(c - I_{h}(c))_{x}|^{2}}{Q_{h}} \, dx \\
&\leq \epsilon \int_{I} \frac{|(c - c_{h})_{x}|^{2}}{Q_{h}} \, dx + C_{\epsilon} \|c\|_{H^{2} (I)}^{2} h^{2}, \displaybreak[0] \\
|K_6| &\leq \epsilon \int_{I} \frac{|(c - c_{h})_{x}|^{2}}{Q_{h}} \, dx + C_{\epsilon} \| e_{ux} \|_{L^{2}(I)}^{2} + C_{\epsilon} h^{2}, \displaybreak[0] \\
|K_7| &\leq C\|c\|_{H^{2} (I)}^{2} h^{2} + C \| e_{ux} \|_{L^{2}(I)}^{2} + C h^{2}.
\end{align*}
Summarizing all these estimates and using \eqref{(3.2)} we obtain from \eqref{eq:toest_c} that 
\begin{align*}
\frac{d}{dt} \Big{(} \int_{I} & \frac{1}{2} |c - c_{h}|^{2} |Q_{h}| \, dx \Big{)} + \int_{I} \frac{|c_{x} - c_{hx}|^{2}}{Q_{h}} \, dx \nonumber \\
&\leq \epsilon C \int_{I} \frac{|c_{x} - c_{hx}|^{2}}{Q_{h}} \, dx \nonumber \displaybreak[0] \\
&\quad + \frac{d}{dt} \Big{(} \int_{I} (c - c_{h}) Q_{h} (c - I_{h}(c)) \, dx + \int_{I} c (Q - Q_{h}) (c - I_{h}(c)) \, dx \Big{)} \displaybreak[0] \nonumber \\
&\quad + C \int_{I} |c - c_{h}|^{2} Q_{h} \, dx + C_{\epsilon} \|e_{ux}\|_{L^{2}(I)}^{2} + C_{\epsilon} \|e_{ut}\|_{L^{2}(I)}^{2} \\
& \quad + C_{\epsilon} \|e_{ux}\|_{L^{2}(I)}^{2} \|e_{utx}\|_{L^{2}(I) }^{2}
+ C_{\epsilon} h^{2} \|e_{utx}\|_{L^{2}(I)}^{2} + C_{\epsilon} h^{2}.
\end{align*}
Integrating with respect to time from $0$ to ${\intertime}$, using \eqref{eq:estK3}, \eqref{(3.2)}, and embedding theory we get for $\epsilon$ small enough that
\begin{align*}
\int_{I} |c({\intertime}) - & c_{h}({\intertime})|^{2} \, dx + \int_0^{\intertime} \int_{I} |c_{x} - c_{hx}|^{2} \, dx dt \nonumber \\
&\leq C \int_{I} |\cic - \cich|^{2} \, dx + \int_{I} |(\cic Q(0) - \cich Q_h(0)) (\cic - I_{h}(\cic)) |\, dx \displaybreak[0] \nonumber \\
&\quad + C \hat{\eps} \int_{I} |c({\intertime}) - c_{h}({\intertime})|^{2} \, dx + C\|e_{ux} ({\intertime})\|_{L^{2}(I)}^{2} \displaybreak[0] \nonumber \\
&\quad + C \int_{0}^{{\intertime}} \int_{I} |c - c_{h}|^{2} \, dx dt + 
C \int_{0}^{{\intertime}} \|e_{ux}\|_{L^{2}(I)}^{2} dt + C \int_{0}^{{\intertime}} \|e_{ut}\|_{L^{2}(I)}^{2} dt \\
& \quad + C \int_{0}^{{\intertime}} \|e_{ux}\|_{L^{2}(I)}^{2} \|e_{utx}\|_{L^{2}(I) }^{2} dt
+ C h^{2} \int_{0}^{{\intertime}} \|e_{utx}\|_{L^{2}(I)}^{2} dt + C_{\hat{\eps}} h^{2}.
\end{align*}
Note that thanks to our choice of the discrete initial data \eqref{icondh} 
\[
 \int_{I} |\cic - \cich|^{2} \, dx = \int_{I} |\cic - I_{h}(\cic)|^{2} \, dx \leq C\|\cic\|_{H^1(I)}^{2} h^{2}.
\]
Moreover with the arguments used to estimate $K_3$, and using the fact that 
$\|e_{ux}(0)\|_{L^2(I)}^{2} = 0$ (recall Lemma~\ref{indata}) 
we get that
\begin{align*}
\int_{I} |(\cic Q(0) - \cich Q_h(0)) (\cic - I_{h}(\cic)) |\, dx \leq C\|\cic\|_{H^1(I)}^{2} h^{2} + C h^{2} + C\|e_{ux}(0)\|_{L^{2}(I)}^{2} \leq Ch^{2}.
\end{align*}
Choosing $\hat{\eps}$ small enough and using the above estimates for the initial data finishes the proof.
\end{proof}

\subsection{Proof of the main Theorem}

From Lemma~\ref{lemma:c-ch} and Lemma~\ref{lemma3.2}, and then using \eqref{(3.21)} and \eqref{(3.21)bis} we infer for ${\intertime} \in [0, T_{h})$ that
\begin{align*}
\| (c - & c_{h})({\intertime})\|_{L^{2}(I)}^{2} + \int_0^{\intertime} \|(c - c_{h})_{x}\|_{L^{2}(I)}^{2} dt \nonumber \\
& \leq C \epsilon \| e_{w}(\intertime) \|_{L^{2}(I)}^{2} + C_{\epsilon} \| e_{u}(\intertime) \|_{L^{2}(I)}^{2} + C \int_{0}^{{\intertime}} \| c - c_{h} \|_{L^{2}(I)}^{2} dt \nonumber \\
& \quad + C \int_{0}^{{\intertime}} \big{(} \epsilon \| e_{w} \|_{L^{2}(I)}^{2} + C_{\epsilon} \| e_{u} \|_{L^{2}(I)}^{2} \big{)} dt + C \int_{0}^{{\intertime}} \|e_{ut}\|_{L^{2}(I)}^{2} dt \\
& \quad + C \int_{0}^{{\intertime}} \big{(} \epsilon \| e_{w} \|_{L^{2}(I)}^{2} + C_{\epsilon} \| e_{u} \|_{L^{2}(I)}^{2} + C h^{4} |\log h|^{2} \big{)} \|e_{utx}\|_{L^{2}(I) }^{2} dt \nonumber \\
& \quad + C h^{2} \int_{0}^{{\intertime}} \|e_{utx}\|_{L^{2}(I)}^{2} dt + C h^{2} \nonumber \\
& \leq C \int_{0}^{{\intertime}} \|c - c_{h}\|^{2}_{L^{2}(I)} dt + C h^{2} + C h^{2} \int_{0}^{{\intertime}} \|e_{utx}\|_{L^{2}(I)}^{2} dt \\
& \quad 
+ C \int_{0}^{{\intertime}} \left( \int_{0}^{t}\|(c - c_{h})(s)\|^{2}_{L^{2}(I)} ds \right) \|e_{utx}(t)\|_{L^{2}(I) }^{2} dt \displaybreak[0] \\
& \leq C \int_{0}^{{\intertime}} \|c - c_{h}\|^{2}_{L^{2}(I)} dt + C h^{2} + C h^{2} \int_{0}^{{\intertime}} \|e_{utx}\|_{L^{2}(I)}^{2} dt\\
& \quad 
+ C \left(\int_{0}^{{\intertime}} \|(c - c_{h})(t)\|^{2}_{L^{2}(I)} dt \right) \left(\int_{0}^{{\intertime}} \|e_{utx}(t)\|_{L^{2}(I) }^{2} dt \right) \displaybreak[0] \\
& \leq C \int_{0}^{{\intertime}} \|c - c_{h}\|^{2}_{L^{2}(I)} dt + C h^{2} + C \frac{h^{2}}{\mu(h)} \left( \int_{0}^{{\intertime}} \|c - c_{h}\|^{2}_{L^{2}(I)} dt + h^{2}\right)\\
& \quad +C\frac{1}{\mu(h)} \left(\int_{0}^{{\intertime}} \|c - c_{h}\|^{2}_{L^{2}(I)} dt \right) \left(\int_{0}^{{\intertime}} \|c - c_{h}\|^{2}_{L^{2}(I)} dt + h^{2} \right).
\end{align*}

Using that $\mu(h) \sim h^{r}$ with $r \in [1,2)$ and the Cauchy-Schwarz inequality we finally obtain for any~${\intertime} \in [0,T_{h})$ that
\begin{align} \label{zwischenergebnis}
\| (c - & c_{h})({\intertime})\|_{L^{2}(I)}^{2} + \int_0^{\intertime} \|(c - c_{h})_{x}\|_{L^{2}(I)}^{2} dt \nonumber \\
& \leq C \int_{0}^{{\intertime}} \|c - c_{h}\|^{2}_{L^{2}(I)} dt + C h^{2} + \frac{C}{\mu(h)} \int_{0}^{{\intertime}} \|(c - c_{h})\|^4_{L^{2}(I)} dt.
\end{align}

We now employ the following generalized Gronwall lemma, whose proof can be found in \cite[Prop.~6.2]{Bartels}:
\begin{lemma}\label{genG}
Suppose that the nonnegative functions $a$ and $y_{i}$, $i=1,2,3$ with $y_{1} \in C([0,\bar{T}])$, $y_{2}, y_{3} \in L^{1}(0,\bar{T})$, $a \in L^{\infty}(0,\bar{T})$, and the real number $A \geq 0$ satisfy
\begin{align*}
y_{1}(T') + \int_{0}^{T'} y_{2}(t) dt \leq A + \int_{0}^{T'} a(t)y_{1}(t) dt + \int_{0}^{T'} y_{3}(t) dt
\end{align*}
for all $T' \in [0,\bar{T}]$. Assume that for some $B\geq0$, some $\beta>0$, and every $T' \in [0,\bar{T}]$, we have that
\begin{align*}
\int_{0}^{T'} y_{3}(t) dt \leq B \left( \sup_{ t\in [0,T']} y_{1}^{\beta}(t)\right) \int_{0}^{T'} (y_{1}(t) + y_{2}(t)) dt.
\end{align*}
Set $E:=\exp (\int_{0}^{\bar{T}} a(t) dt)$ and assume that 
\begin{align}\label{condAE}
8AE \leq \frac{1}{(8B(1+\bar{T})E)^{1 / \beta}}.
\end{align}
We then have
\begin{align*}
\sup_{ t\in [0,\bar{T}]} y_{1}(t) + \int_{0}^{\bar{T}} y_{2}(t) dt \leq 8AE= 8A \exp (\int_{0}^{\bar{T}} a(t) dt).
\end{align*}
\end{lemma}

In our situation we take $\bar{T} = \intertime$, $y_{1}(t)= \| (c-c_{h})(t)\|_{L^{2}(I)}^{2}$, $A=Ch^{2}$ where $C$ is the constant from \eqref{zwischenergebnis} (which depends on $u,c,T$ but not on $h$ or $T_{h}$), $a(t)=C$, $B=\frac{C}{\mu(h)}$, $y_{3} = \frac{C}{\mu(h)} y_{1}^{2}$, $\beta=1$, $y_{2}=0$. 
For $0 < {\intertime} < T_{h} \leq T$ we see that $8 A E = 8Ch^{2} \exp(C \intertime) \leq 8Ch^{2} \exp(C T)$ and that
\[
 \frac{1}{(8B(1+\bar{T})E)^{1 / \beta}} = \frac{\mu(h)}{8C (1+{\intertime}) \exp (C{\intertime})} \geq \frac{\mu(h)}{8C (1+T) \exp (CT)}.
\]
With our choice \eqref{def:mu} for $\mu(h)$ where $r<2$ we get that \eqref{condAE} is satisfied for all $h \leq h_{0}$ if
\[
 8 C h_{0}^{2} \exp(C T) \leq \frac{C_{\mu} h_{0}^r}{8C (1+T) \exp (CT)} \quad \Leftrightarrow \quad  h_{0}^{2-r} \leq \frac{C_{\mu}}{64 C^2 (1+T) \exp(2 C T)}.
\]

Thus we infer that for $h \leq h_{0}$ and any ${\intertime} \in [0, T_{h})$ (and, by continuity, in fact up to time $T_{h}$)
\begin{align} \label{errC}
\| (c - & c_{h})({\intertime})\|_{L^{2}(I)}^{2} + \int_0^{\intertime} \|(c - c_{h})_{x}\|_{L^{2}(I)}^{2} dt \leq C h^{2}.
\end{align}
Plugging this result back into \eqref{(3.21)}, \eqref{(3.21)bis}, and using Lemma~\ref{lemma3.2}, we obtain for any ${\intertime} \in [0, T_{h}]$ that
\begin{align} \label{errE}
\| e_{w}({\intertime})\|_{L^{2}(I)}^{2} + \| e_{u}({\intertime})\|_{L^{2}(I)}^{2} + \| e_{ux}({\intertime})\|_{L^{2}(I)}^{2} 
+ \int_0^{\intertime} \| e_{ut}({\intertime})\|_{L^{2}(I)}^{2} dt + \mu(h) \int_{0}^{\bar{t}} \|e_{utx} \|_{L^{2}(I)}^{2} dt \leq Ch^{2}.
\end{align}

Now that we have achieved error estimates on the time intervall $[0, T_{h}]$ with a constant $C$ that does not depend on $h$ or $T_{h}$ we are able to show that in fact it must be $T_{h}=T$ for all $h$ sufficiently small. Indeed, observe that by \eqref{pippoQ}, \eqref{(2.19)}, \eqref{IE-2}, \eqref{errE}, we get
\begin{align*}
\|Q_{h}({\intertime}) \|_{L^{\infty}(I)} &\leq C_{0} + \|(Q-Q_{h})({\intertime}) \|_{L^{\infty}(I)} \leq C_{0} + \|\rho_{ux}({\intertime})\|_{L^{\infty}(I)}
+ \|e_{ux}({\intertime})\|_{L^{\infty}(I)} \\& \leq C_{0} +C h |\log h| + C\frac{h}{\sqrt{h}} \leq \frac{3}{2}C_{0}
\end{align*}
provided that $h \leq h_{0}$ (after decreasing $h_{0}$ if required). Similarly, by \eqref{(2.24)}, \eqref{IE-2}, \eqref{(4.2)}, \eqref{errE}, and \eqref{errC}, we obtain
\begin{align*}
\|w_{h}({\intertime}) \|_{L^{\infty}(I)} & \leq \|e_{w} \|_{L^{\infty}(I)} + \| \what -I_{h} w \|_{L^{\infty}(I)} + \|I_{h} w \|_{L^{\infty}(I)}\\
& \leq C \frac{h}{\sqrt{h}} + \frac{C}{\sqrt{h}} \| \what -I_{h} w \|_{L^{2}(I)} + C_{1} \\
& \leq C_{1}+ C \frac{h}{\sqrt{h}} + \frac{C}{\sqrt{h}} ( \| \rho_{w} \|_{L^{2}(I)} +\| w -I_{h} w \|_{L^{2}(I)} ) \leq \frac{3}{2}C_{1} ,\\
\|c_{h}({\intertime}) \|_{L^{\infty}(I)} &\leq \|I_{h} c({\intertime}) \|_{L^{\infty}(I)} + \| (c_{h} - I_{h} c)({\intertime}) \|_{L^{\infty}(I)} \\
& \leq \| c \|_{C([0,T], L^{\infty }(I))} +\frac{C}{\sqrt{h}} \| (c_{h} - I_{h} c )({\intertime})\|_{L^{2}(I)}\\
& \leq \embconst \| c \|_{C([0,T], H^{1}(I))} +\frac{C}{\sqrt{h}} ( \| (c_{h} - c)({\intertime}) \|_{L^{2}(I)} + \| (c - I_{h} c)({\intertime}) \|_{L^{2}({\intertime})})\\
& \leq \embconst C_{2} +\frac{C}{\sqrt{h}} (h + h\| c \|_{C([0,T], H^{1}(I))}) \leq \frac{3}{2} \embconst C_{2},\\
\|c_{h}\|_{L^{2}((0,T_{h}),H^1(I))} &\leq \frac{3}{2} C_{3},
\end{align*}
for all $h \leq h_{0}$ independently of $T_{h}$ (after decreasing $h_{0}$ if required). If we had that $T_{h} < T$ then we could establish \eqref{(3.2)} and \eqref{(3.2)bis} on the time intervall $[0, T_{h} + \delta]$ for some $\delta >0$ which would contradict the maximality of $T_{h}$. Hence $T_{h}=T$. The first three error estimates stated in Theorem~\ref{mainthm} follow from \eqref{errE}, \eqref{errC}, \eqref{(2.18)}, \eqref{(2.20)}, \eqref{(2.24)}, Lemma~\ref{lemma3.3}, \eqref{(3.2)bis}, \eqref{stern}, and \eqref{(2.23)}. The last statement in Theorem~\ref{mainthm} follows from \eqref{errE} and \eqref{(2.21)}.

\section{Numerical simulations}
\label{sec:numsim}

We now aim for assessing and supporting our theoretical convergence results by some numerical simulations. We prescribe functions $(u,w,c)$ and ensure that they solve \eqref{eq1}--\eqref{icond} by accounting for suitable source terms $s_u, s_c : I \to \mbR$ for $u$ and $c$, respectively. 

The time discretisation of Problem \ref{prob:semidis} is based on uniform time steps $\delta = h^2$. This choice turned out small enough to ensure that the errors that we report on below are purely due to the spatial discretisation. An upper index will indicate values at the time $t^{(m)} := m \delta$ in the following, $m = 0 \dots, M := T / \delta$. We use a simple order-one IMEX-scheme which linearises the problem in each time step and decouples the solution of the geometric equation from the solution of the equation on the curve: 

\begin{prob}[Fully discrete scheme] \label{prob:fullydis}
Find functions $u_{\delta h}^{(m)} \in X_h$ and $w_{\delta h}^{(m)}, c_{\delta h}^{(m)} \in X_{h0}$, $m=0, \dots, M$, of the form
\begin{equation*}
u_{\delta h}^{(m)} = \sum_{j=0}^{N} u_j^{(m)} \varphi_j(x), \quad c_{\delta h}^{(m)} = \sum_{j=1}^{N-1} c_j^{(m)} \varphi_j(x), \quad w_{\delta h}^{(m)} = \sum_{j=1}^{N-1} w_j^{(m)} \varphi_j(x),
\end{equation*} 
with $u_j^{(m)}, c_{j}^{(m)}, w_{j}^{(m)} \in \mbR$ such that 
\[
 u_{\delta h}^{(m)} -I_{h}(\ubc) \in X_{h0} \, \, \forall m = 1, \dots, M, \quad u_{\delta h}^{(0)} = \hat{\uic}, \quad c_{\delta h}^{(0)} = I_{h} (\cic),
\]
and such that 
\begin{align} \label{eq:fulldis_u}
\int_I \mu(h) \frac{u_{\delta h x}^{(m)} - u_{\delta h x}^{(m-1)}}{\delta} \varphi_{hx} &+ \frac{(u_{\delta h}^{(m)} - u_{\delta h}^{(m-1)}) \varphi_{h}}{\delta Q_{\delta h}^{(m-1)}} + \frac{1}{2} (w_{\delta h}^{(m-1)})^2 \frac{u_{\delta h x}^{(m)} \varphi_{hx}}{(Q_{\delta h}^{(m-1)})^3} + \frac{w_{\delta h x}^{(m)} \varphi_{hx}}{(Q_{\delta h}^{(m-1)})^3} \, dx \nonumber \\ 
&= \int_{I} I_{h} (f(c_{\delta h}^{(m-1)}) + s_u^{(m)}) \varphi_{h} \, dx, \\ \label{eq:fulldis_w}
\int_I \frac{w_{\delta h}^{(m)} \psi_{h}}{Q_{\delta h}^{(m-1)}} - \frac{u_{\delta h x}^{(m)} \psi_{hx}}{Q_{\delta h}^{(m-1)}} \, dx \, &= 0, \\  \label{eq:fulldis_c}
\int_{I} c_{\delta h}^{(m)} Q_{\delta h}^{(m)} \zeta_{h} + \delta \frac{c_{\delta h x}^{(m)} \zeta_{hx}}{Q_{\delta h}^{(m)}} \, dx \, &= \int_{I} c_{\delta h}^{(m-1)} Q_{\delta h}^{(m-1)} \zeta_{h} + \delta I_{h}(s_c^{(m)}) \zeta_{h}, 
\end{align}
for all $\varphi_{h}, \psi_{h}, \zeta_{h} \in X_{h0}$ and for $m=1, \dots, M$. Here, $Q_{\delta h}^{(m-1)}$ denotes the discrete length element, $Q_{\delta h}^{(m-1)} = \sqrt{1 + (u_{\delta h x}^{(m-1)})^{2}}$, and $\mu(h)=C_\mu h^{r}$ for some $r \in [1,2)$ and $C_\mu \geq 0$ (as defined in \eqref{def:mu}).
\end{prob}

In the test examples further below we monitored the following errors:
\begin{align}
\mmm{E}_u(L^\infty,L^2) &:= \| u - u_{\delta h} \|_{L^\infty(J,L^2(I))}^2, &\quad \mmm{E}_u(L^\infty,H^1) &:= \| u_{x} - u_{\delta h x} \|_{L^\infty(J,L^2(I))}^2, \nonumber \\
\mmm{E}_u(H^1,L^2) &:= \| u_{t} - u_{\delta h t} \|_{L^2(J,L^2(I))}^2,     &\quad \mmm{E}_u(H^1,H^1) &:= \| u_{t x} - u_{\delta h t x} \|_{L^2(J,L^2(I))}^2, \nonumber \\
\mmm{E}_w(L^\infty,L^2) &:= \| w - w_{\delta h} \|_{L^\infty(J,L^2(I))}^2, &\quad \mmm{E}_w(L^2,H^1) &:= \| w_{x} - w_{\delta h x} \|_{L^2(J,L^2(I))}^2, \nonumber \\
\mmm{E}_c(L^\infty,L^2) &:= \| c - c_{\delta h} \|_{L^\infty(J,L^2(I))}^2, &\quad \mmm{E}_c(L^2,H^1) &:= \| c_{x} - c_{\delta h x} \|_{L^2(J,L^2(I))}^2, \label{eq:err_defs} 
\end{align}
where $J = (0,T)$ and $u_{\delta h}$ has been extended by linearly interpolating on each time interval so that, for instance, $u_{\delta h t} = (u_{\delta h}^{(m)} - u_{\delta h}^{(m-1)}) / \delta$ for $t \in (t^{(m-1)},t^{(m)})$. We used sufficiently accurate quadrature rules on each rectangle $[t^{(m-1)},t^{(m)}] \times [x_{j-1},x_{j}]$. 


\setlength{\tabcolsep}{7pt}
\begin{table}
{\footnotesize
\begin{tabular}{|r||l|l||l|l||l|l||l|l|} \hline 
$N$ & $\mmm{E}_u(L^\infty,L^2)$ & $\eoc$ & $\mmm{E}_u(L^\infty,H^1)$ & $\eoc$ & $\mmm{E}_u(H^1,L^2)$ & $\eoc$ & $\mmm{E}_u(H^1,H^1)$ & $\eoc$\\ \hline \hline 
   61 & 5.454e-06 & --     & 5.701e-05 & --     & 8.398e-05 & --     & 9.160e-04 & --     \\
   81 & 3.258e-06 & 1.7910 & 3.389e-05 & 1.8080 & 5.271e-05 & 1.6193 & 5.696e-04 & 1.6516 \\
  101 & 2.155e-06 & 1.8524 & 2.236e-05 & 1.8637 & 3.600e-05 & 1.7087 & 3.870e-04 & 1.7315 \\
  131 & 1.310e-06 & 1.8967 & 1.357e-05 & 1.9042 & 2.258e-05 & 1.7777 & 2.417e-04 & 1.7941 \\
  161 & 8.779e-07 & 1.9283 & 9.081e-06 & 1.9334 & 1.544e-05 & 1.8306 & 1.649e-04 & 1.8426 \\
  201 & 5.683e-07 & 1.9490 & 5.874e-06 & 1.9524 & 1.018e-05 & 1.8680 & 1.085e-04 & 1.8771 \\
\hline
\end{tabular}

\medskip

\begin{tabular}{|r||l|l||l|l||l|l||l|l|} \hline 
$N$ & $\mmm{E}_w(L^\infty,L^2)$ & $\eoc$ & $\mmm{E}_w(L^2,H^1)$ & $\eoc$ & $\mmm{E}_c(L^\infty,L^2)$ & $\eoc$ & $\mmm{E}_c(L^2,H^1)$ & $\eoc$\\ \hline \hline 
   61 & 6.840e-04 & --     & 7.722e-02 & --     & 2.492e-06 & --     & 2.027e-02 & -- \\
   81 & 4.060e-04 & 1.8129 & 4.361e-02 & 1.9864 & 7.992e-07 & 3.9536 & 1.141e-02 & 1.9971 \\
  101 & 2.677e-04 & 1.8660 & 2.796e-02 & 1.9908 & 3.299e-07 & 3.9647 & 7.306e-03 & 1.9984 \\
  131 & 1.624e-04 & 1.9055 & 1.657e-02 & 1.9938 & 1.165e-07 & 3.9690 & 4.324e-03 & 1.9991 \\
  161 & 1.087e-04 & 1.9341 & 1.095e-02 & 1.9958 & 5.108e-08 & 3.9693 & 2.855e-03 & 1.9995 \\
  201 & 7.030e-05 & 1.9530 & 7.013e-03 & 1.9970 & 2.108e-08 & 3.9668 & 1.827e-03 & 1.9997 \\
\hline
\end{tabular}
}
\caption{Errors \eqref{eq:err_defs} and EOCs for the first test problem \eqref{eq:testA} described in Section \ref{sec:numsim} with $\mu(h) = 40 h$.}
\label{tab:conv_ha}
\end{table}

\begin{table}
{\footnotesize
\begin{tabular}{|r||l|l||l|l||l|l||l|l|} \hline 
$N$ & $\mmm{E}_u(L^\infty,L^2)$ & $\eoc$ & $\mmm{E}_u(L^\infty,H^1)$ & $\eoc$ & $\mmm{E}_u(H^1,L^2)$ & $\eoc$ & $\mmm{E}_u(H^1,H^1)$ & $\eoc$\\ \hline \hline 
   61 & 1.233e-05 & --     & 1.289e-04 & --     & 1.731e-04 & --     & 1.882e-03 & --     \\
   81 & 4.828e-06 & 3.2587 & 5.019e-05 & 3.2795 & 7.588e-05 & 2.8672 & 8.146e-04 & 2.9105 \\
  101 & 2.155e-06 & 3.6146 & 2.236e-05 & 3.6239 & 3.600e-05 & 3.3417 & 3.870e-04 & 3.3350 \\
  131 & 7.925e-07 & 3.8129 & 8.215e-06 & 3.8164 & 1.390e-05 & 3.6256 & 1.514e-04 & 3.5777 \\
  161 & 3.517e-07 & 3.9128 & 3.646e-06 & 3.9124 & 6.336e-06 & 3.7851 & 7.033e-05 & 3.6923 \\
  201 & 1.455e-07 & 3.9556 & 1.509e-06 & 3.9517 & 2.674e-06 & 3.8666 & 3.067e-05 & 3.7192 \\
\hline
\end{tabular}

\medskip

\begin{tabular}{|r||l|l||l|l||l|l||l|l|} \hline 
$N$ & $\mmm{E}_w(L^\infty,L^2)$ & $\eoc$ & $\mmm{E}_w(L^2,H^1)$ & $\eoc$ & $\mmm{E}_c(L^\infty,L^2)$ & $\eoc$ & $\mmm{E}_c(L^2,H^1)$ & $\eoc$\\ \hline \hline 
   61 & 1.591e-03 & --     & 8.962e-02 & --     & 2.499e-06 & --     & 2.027e-02 & --     \\
   81 & 6.053e-04 & 3.3586 & 4.604e-02 & 2.3152 & 8.008e-07 & 3.9557 & 1.141e-02 & 1.9971 \\
  101 & 2.677e-04 & 3.6552 & 2.796e-02 & 2.2345 & 3.299e-07 & 3.9738 & 7.306e-03 & 1.9984 \\
  131 & 9.802e-05 & 3.8299 & 1.585e-02 & 2.1631 & 1.160e-07 & 3.9849 & 4.324e-03 & 1.9991 \\
  161 & 4.344e-05 & 3.9198 & 1.023e-02 & 2.1088 & 5.064e-08 & 3.9913 & 2.855e-03 & 1.9995 \\
  201 & 1.796e-05 & 3.9589 & 6.442e-03 & 2.0733 & 2.077e-08 & 3.9947 & 1.827e-03 & 1.9997 \\
\hline
\end{tabular}
}
\caption{Errors \eqref{eq:err_defs} and EOCs for the first test problem \eqref{eq:testA} described in Section \ref{sec:numsim} with $\mu(h) = 4000 h^2$.}
\label{tab:conv_hb}
\end{table}

\begin{table}
{\footnotesize
\begin{tabular}{|r||l|l||l|l||l|l||l|l|} \hline 
$N$ & $\mmm{E}_u(L^\infty,L^2)$ & $\eoc$ & $\mmm{E}_u(L^\infty,H^1)$ & $\eoc$ & $\mmm{E}_u(H^1,L^2)$ & $\eoc$ & $\mmm{E}_u(H^1,H^1)$ & $\eoc$\\ \hline \hline 
   61 & 5.165e-06 & --     & 5.399e-05 & --     & 7.989e-05 & --     & 8.725e-04 & --     \\
   81 & 2.369e-06 & 2.7091 & 2.467e-05 & 2.7229 & 3.906e-05 & 2.4876 & 4.262e-04 & 2.4905 \\
  101 & 1.257e-06 & 2.8391 & 1.307e-05 & 2.8478 & 2.152e-05 & 2.6705 & 2.358e-04 & 2.6524 \\
  131 & 5.856e-07 & 2.9120 & 6.078e-06 & 2.9174 & 1.037e-05 & 2.7845 & 1.146e-04 & 2.7491 \\
  161 & 3.172e-07 & 2.9525 & 3.290e-06 & 2.9557 & 5.728e-06 & 2.8573 & 6.402e-05 & 2.8056 \\
  201 & 1.634e-07 & 2.9729 & 1.694e-06 & 2.9746 & 2.998e-06 & 2.9009 & 3.402e-05 & 2.8334 \\
\hline
\end{tabular}

\medskip

\begin{tabular}{|r||l|l||l|l||l|l||l|l|} \hline 
$N$ & $\mmm{E}_w(L^\infty,L^2)$ & $\eoc$ & $\mmm{E}_w(L^2,H^1)$ & $\eoc$ & $\mmm{E}_c(L^\infty,L^2)$ & $\eoc$ & $\mmm{E}_c(L^2,H^1)$ & $\eoc$\\ \hline \hline 
   61 & 6.469e-04 & --     & 7.675e-02 & --     & 2.492e-06 & --     & 2.027e-02 & --     \\
   81 & 2.942e-04 & 2.7391 & 4.229e-02 & 2.0722 & 7.980e-07 & 3.9581 & 1.141e-02 & 1.9971 \\
  101 & 1.556e-04 & 2.8533 & 2.668e-02 & 2.0632 & 3.289e-07 & 3.9725 & 7.306e-03 & 1.9984 \\
  131 & 7.236e-05 & 2.9189 & 1.557e-02 & 2.0532 & 1.157e-07 & 3.9805 & 4.324e-03 & 1.9991 \\
  161 & 3.917e-05 & 2.9559 & 1.019e-02 & 2.0439 & 5.060e-08 & 3.9850 & 2.855e-03 & 1.9995 \\
  201 & 2.017e-05 & 2.9746 & 6.466e-03 & 2.0363 & 2.078e-08 & 3.9872 & 1.827e-03 & 1.9997 \\
\hline
\end{tabular}
}
\caption{Errors \eqref{eq:err_defs} and EOCs for the first test problem \eqref{eq:testA} described in Section \ref{sec:numsim} with $\mu(h) = 300 h^{3/2}$.}
\label{tab:conv_hc}
\end{table}

\begin{table}
{\footnotesize
\begin{tabular}{|r||l|l||l|l||l|l||l|l|} \hline 
$N$ & $\mmm{E}_u(L^\infty,L^2)$ & $\eoc$ & $\mmm{E}_u(L^\infty,H^1)$ & $\eoc$ & $\mmm{E}_u(H^1,L^2)$ & $\eoc$ & $\mmm{E}_u(H^1,H^1)$ & $\eoc$\\ \hline \hline 
   61 & 3.516e-06 & --     & 3.686e-05 & --     & 5.601e-05 & --     & 6.201e-04 & --     \\
   81 & 2.664e-06 & 0.9642 & 2.773e-05 & 0.9900 & 4.363e-05 & 0.8680 & 4.742e-04 & 0.9325 \\
  101 & 2.155e-06 & 0.9507 & 2.235e-05 & 0.9649 & 3.599e-05 & 0.8624 & 3.870e-04 & 0.9104 \\
  131 & 1.679e-06 & 0.9503 & 1.738e-05 & 0.9586 & 2.862e-05 & 0.8729 & 3.048e-04 & 0.9091 \\
  161 & 1.377e-06 & 0.9544 & 1.424e-05 & 0.9594 & 2.381e-05 & 0.8874 & 2.521e-04 & 0.9150 \\
  201 & 1.111e-06 & 0.9596 & 1.149e-05 & 0.9629 & 1.947e-05 & 0.9011 & 2.052e-04 & 0.9229 \\
\hline
\end{tabular}

\medskip

\begin{tabular}{|r||l|l||l|l||l|l||l|l|} \hline 
$N$ & $\mmm{E}_w(L^\infty,L^2)$ & $\eoc$ & $\mmm{E}_w(L^2,H^1)$ & $\eoc$ & $\mmm{E}_c(L^\infty,L^2)$ & $\eoc$ & $\mmm{E}_c(L^2,H^1)$ & $\eoc$\\ \hline \hline 
   61 & 4.375e-04 & --     & 7.418e-02 & --     & 2.489e-06 & --     & 2.026e-02 & --     \\
   81 & 3.312e-04 & 0.9674 & 4.271e-02 & 1.9185 & 7.984e-07 & 3.9525 & 1.141e-02 & 1.9970 \\
  101 & 2.677e-04 & 0.9539 & 2.796e-02 & 1.8988 & 3.299e-07 & 3.9604 & 7.305e-03 & 1.9984 \\
  131 & 2.084e-04 & 0.9537 & 1.709e-02 & 1.8752 & 1.167e-07 & 3.9596 & 4.323e-03 & 1.9991 \\
  161 & 1.708e-04 & 0.9577 & 1.164e-02 & 1.8478 & 5.138e-08 & 3.9522 & 2.854e-03 & 1.9995 \\
  201 & 1.378e-04 & 0.9627 & 7.764e-03 & 1.8181 & 2.133e-08 & 3.9389 & 1.827e-03 & 1.9997 \\
\hline
\end{tabular}
}
\caption{Errors \eqref{eq:err_defs} and EOCs for the first test problem \eqref{eq:testA} described in Section \ref{sec:numsim} with $\mu(h) = 4 h^{1/2}$.}
\label{tab:conv_hd}
\end{table}

In a first example, let $T=1$ and 
\[
 f(c) = \frac{1-2c}{10},
\]
and consider
\begin{equation} \label{eq:testA}
 \begin{split}
 u(x,t) &= \frac{5}{2} \cos(2 \pi t) (x-1)^3 x^5, \\
 c(x,t) &= \frac{1}{10} \sin(7 \pi x) \sin(4 \pi t).  
 \end{split}
\end{equation}
The source functions $s_u(x,t)$ and $s_c(x,t)$ are picked such that the above functions solve \eqref{eq1}--\eqref{icond}. Note that then $\ubc=0$. We remark that the function $u$ has also been considered in \cite{DD06}. 

For varying values of $N$ ($h = 1/N$, $\delta = h^2$) the errors and corresponding $\eoc$'s are displayed in Table \ref{tab:conv_ha} for the choice $\mu(h) = 40 h$, i.e., $r = 1$. 
For most errors we observe EOCs close to two (those for $u_{t}$ and $u_{tx}$ are a bit smaller but still increasing). This corresponds to linear convergence as predicted in Theorem \ref{mainthm} except for $u_{tx}$. In that case we only could show a rate of $2-r =1$ but observe a better convergence behaviour. Regarding the error of $c$ in the norm $L^\infty((0,T),L^2(I))$ we also observe faster (here quadratic) convergence. 

We have also carried out computations with $\mu(h) = 4000 h^2$ for comparison. Recall that this case $r=2$ is not covered by the theory but we didn't observe any issues with solving the discrete problems. The results are displayed in Table \ref{tab:conv_hb}. We notice faster (quadratic) convergence of all errors except for $\mmm{E}_w(L^2,H^1)$ and $\mmm{E}_c(L^2,H^1)$ where linear convergence is measured. We also see that $\mmm{E}_c(L^\infty,L^2)$ and $\mmm{E}_c(L^2,H^1)$ barely change. 

For further comparison, we chose $\mu(h) = 300 h^r$ with an intermediate growth rate of $r = \frac{3}{2}$ and with $r =\frac{1}{2}$, see Tables \ref{tab:conv_hc} and \ref{tab:conv_hd} for the results, respectively. The findings are consistent in the sense that the EOCs for all fields except for $\mmm{E}_w(L^2,H^1)$, $\mmm{E}_c(L^\infty,L^2)$, and $\mmm{E}_c(L^2,H^1)$ are close to three or one now, indicating convergence orders of $\frac{3}{2}$ or $\frac{1}{2}$, respectively. Again, the errors of $c$ are very close to those in the other two simulation test series. In the case $r = \frac{1}{2}$ we even observe an impact on the EOCs for $\mmm{E}_w(L^2,H^1)$, namely a dip away from two.

The super-convergence of $\mmm{E}_u(L^\infty,H^1)$ for $r > 1$ is a bit surprising. The fact that the errors of $c$ barely depends on the scaling of $\mu$ in $h$ indicates that the geometric error has a smaller influence than the approximation of the diffusion term and the data for $c$. In order to investigate these findings a bit further we consider a second example with a more oscillating geometry and less oscillations in the field on the curve. 


\begin{table}
{\footnotesize
\begin{tabular}{|r||l|l||l|l||l|l||l|l|} \hline 
$N$ & $\mmm{E}_u(L^\infty,L^2)$ & $\eoc$ & $\mmm{E}_u(L^\infty,H^1)$ & $\eoc$ & $\mmm{E}_u(H^1,L^2)$ & $\eoc$ & $\mmm{E}_u(H^1,H^1)$ & $\eoc$\\ \hline \hline 
   61 & 9.579e-07 & --     & 4.753e-05 & --     & 1.923e-05 & --     & 1.019e-03 & --     \\
   81 & 4.014e-07 & 3.0228 & 2.497e-05 & 2.2366 & 8.069e-06 & 3.0185 & 5.157e-04 & 2.3688 \\
  101 & 2.089e-07 & 2.9259 & 1.529e-05 & 2.1989 & 4.203e-06 & 2.9231 & 3.085e-04 & 2.3021 \\
  131 & 1.006e-07 & 2.7857 & 8.679e-06 & 2.1582 & 2.029e-06 & 2.7745 & 1.714e-04 & 2.2405 \\
  161 & 5.824e-08 & 2.6327 & 5.588e-06 & 2.1207 & 1.181e-06 & 2.6078 & 1.088e-04 & 2.1881 \\
  201 & 3.337e-08 & 2.4959 & 3.504e-06 & 2.0913 & 6.824e-07 & 2.4579 & 6.737e-05 & 2.1483 \\
  251 & 1.966e-08 & 2.3706 & 2.209e-06 & 2.0665 & 4.064e-07 & 2.3225 & 4.203e-05 & 2.1149 \\
  301 & 1.298e-08 & 2.2758 & 1.520e-06 & 2.0490 & 2.710e-07 & 2.2230 & 2.870e-05 & 2.0906 \\
\hline
\end{tabular}

\medskip

\begin{tabular}{|r||l|l||l|l||l|l||l|l|} \hline 
$N$ & $\mmm{E}_w(L^\infty,L^2)$ & $\eoc$ & $\mmm{E}_w(L^2,H^1)$ & $\eoc$ & $\mmm{E}_c(L^\infty,L^2)$ & $\eoc$ & $\mmm{E}_c(L^2,H^1)$ & $\eoc$\\ \hline \hline 
   61 & 6.560e-04 & --     & 2.516e+00 & --     & 1.648e-08 & --     & 1.362e-04 & --     \\
   81 & 2.936e-04 & 2.7936 & 1.417e+00 & 1.9946 & 5.259e-09 & 3.9698 & 7.650e-05 & 2.0067 \\
  101 & 1.691e-04 & 2.4735 & 9.080e-01 & 1.9969 & 2.176e-09 & 3.9538 & 4.890e-05 & 2.0050 \\
  131 & 9.330e-05 & 2.2666 & 5.375e-01 & 1.9981 & 7.756e-10 & 3.9328 & 2.891e-05 & 2.0038 \\
  161 & 5.971e-05 & 2.1489 & 3.549e-01 & 1.9988 & 3.446e-10 & 3.9065 & 1.907e-05 & 2.0029 \\
  201 & 3.746e-05 & 2.0891 & 2.272e-01 & 1.9992 & 1.451e-10 & 3.8756 & 1.220e-05 & 2.0023 \\
  251 & 2.369e-05 & 2.0531 & 1.454e-01 & 1.9995 & 6.166e-11 & 3.8359 & 7.805e-06 & 2.0018 \\
  301 & 1.635e-05 & 2.0335 & 1.009e-01 & 1.9996 & 3.088e-11 & 3.7920 & 5.418e-06 & 2.0015 \\
\hline
\end{tabular}
}
\caption{Errors \eqref{eq:err_defs} and EOCs for the second test problem \eqref{eq:testB} described in Section \ref{sec:numsim} with $\mu(h) = 40 h$.}
\label{tab:Bconv_ha}
\end{table}

\begin{table}
{\footnotesize
\begin{tabular}{|r||l|l||l|l||l|l||l|l|} \hline 
$N$ & $\mmm{E}_u(L^\infty,L^2)$ & $\eoc$ & $\mmm{E}_u(L^\infty,H^1)$ & $\eoc$ & $\mmm{E}_u(H^1,L^2)$ & $\eoc$ & $\mmm{E}_u(H^1,H^1)$ & $\eoc$\\ \hline \hline 
  61 & 1.118e-06 & --     & 4.969e-05 & --     & 1.912e-05 & --     & 1.024e-03 & --     \\
  81 & 4.591e-07 & 3.0951 & 2.553e-05 & 2.3150 & 8.466e-06 & 2.8317 & 5.183e-04 & 2.3686 \\
 101 & 2.089e-07 & 3.5271 & 1.529e-05 & 2.2974 & 4.203e-06 & 3.1380 & 3.085e-04 & 2.3251 \\
 131 & 7.790e-08 & 3.7608 & 8.543e-06 & 2.2185 & 1.783e-06 & 3.2668 & 1.705e-04 & 2.2594 \\
 161 & 3.482e-08 & 3.8779 & 5.477e-06 & 2.1409 & 8.950e-07 & 3.3211 & 1.080e-04 & 2.1986 \\
 201 & 1.448e-08 & 3.9317 & 3.436e-06 & 2.0892 & 4.234e-07 & 3.3541 & 6.680e-05 & 2.1546 \\
 251 & 5.983e-09 & 3.9616 & 2.172e-06 & 2.0545 & 1.987e-07 & 3.3892 & 4.161e-05 & 2.1215 \\
 301 & 2.897e-09 & 3.9768 & 1.499e-06 & 2.0347 & 1.065e-07 & 3.4197 & 2.837e-05 & 2.1007 \\
\hline
\end{tabular}

\medskip

\begin{tabular}{|r||l|l||l|l||l|l||l|l|} \hline 
$N$ & $\mmm{E}_w(L^\infty,L^2)$ & $\eoc$ & $\mmm{E}_w(L^2,H^1)$ & $\eoc$ & $\mmm{E}_c(L^\infty,L^2)$ & $\eoc$ & $\mmm{E}_c(L^2,H^1)$ & $\eoc$\\ \hline \hline 
  61 & 1.248e-03 &        & 2.581e+00 &        & 1.696e-08 &        & 1.362e-04 &        \\
  81 & 4.060e-04 & 3.9038 & 1.429e+00 & 2.0542 & 5.329e-09 & 4.0250 & 7.650e-05 & 2.0067 \\
 101 & 1.691e-04 & 3.9251 & 9.080e-01 & 2.0336 & 2.176e-09 & 4.0130 & 4.890e-05 & 2.0050 \\
 131 & 5.997e-05 & 3.9514 & 5.343e-01 & 2.0211 & 7.608e-10 & 4.0065 & 2.891e-05 & 2.0038 \\
 161 & 2.629e-05 & 3.9705 & 3.517e-01 & 2.0133 & 3.313e-10 & 4.0037 & 1.907e-05 & 2.0029 \\
 201 & 1.081e-05 & 3.9813 & 2.246e-01 & 2.0087 & 1.356e-10 & 4.0026 & 1.220e-05 & 2.0023 \\
 251 & 4.441e-06 & 3.9883 & 1.436e-01 & 2.0056 & 5.552e-11 & 4.0021 & 7.805e-06 & 2.0018 \\
 301 & 2.144e-06 & 3.9923 & 9.966e-02 & 2.0037 & 2.676e-11 & 4.0019 & 5.418e-06 & 2.0015 \\
\hline
\end{tabular}
}
\caption{Errors \eqref{eq:err_defs} and EOCs for the second test problem \eqref{eq:testA} described in Section \ref{sec:numsim} with $\mu(h) = 4000 h^2$.}
\label{tab:Bconv_hb}
\end{table}

Keeping $T=1$ and $f(c)$ as before consider 
\begin{equation} \label{eq:testB}
 \begin{split}
 u(x,t) &= \frac{5}{2} \cos(2 \pi t) (x-1)^3 x^5 \sin(4 \pi x), \\
 c(x,t) &= \frac{1}{10} \sin(2 \pi x) \sin(\pi t),
 \end{split}
\end{equation}
and choose the source terms again as appropriate to ensure that this is a solution to \eqref{eq1}--\eqref{icond}. 

The errors for $\mu(h) = 40 h$ are displayed in Table \ref{tab:Bconv_ha} whilst those for $\mu(h) = 4000 h^2$ are in Table \ref{tab:Bconv_hb}. We now indeed observe EOCs of around two for both $\mmm{E}_u(L^\infty,H^1)$ and $\mmm{E}_u(H^1,H^1)$ as expected. The behaviour of the other errors is as before.

\section{Conclusion and outlook}

We analysed the semi-discrete scheme \eqref{(2.10)}--\eqref{ch-gleichung} and quantified convergence to the solution of \eqref{(2.4)}--\eqref{c-gleichung}, see Theorem \ref{mainthm} on page \pageref{mainthm}. In order to be able to derive an error estimate for $c_{h}$ a better control of the velocity $u_{ht}$ was required. For this purpose we augmented the geometric equation \eqref{(2.4)} in the semi-discrete scheme with a penalty term, which is a weighted $H^1$ inner product of the velocity with the test function. The weight $\mu(h) \sim h^r$, $r \in [1,2)$ has an  impact on the convergence rates. In turn, the scheme proved quite stable for penalty terms beyond the regime that was analysed. In particular, when $r=2$ was chosen then maximal convergence rates were obtained as one may expect them for the choice of finite elements. This case is not covered by the analysis as then the argument with the generalised Gronwall lemma \ref{genG} fails. On the other hand, the restriction $r \geq 1$ is clearly motivated by the inequality \eqref{ieq:mu_critical}. It was observed in simulations that choosing $r<1$ indeed destroys the order of convergence proved in Theorem \ref{mainthm}.

We make a few remarks on the context of the problem and possible generalisations of the results:
\begin{itemize}
 \item 
 Well-posedness and regularity of the above problem is, to our knowledge, an open problem. We have decided not to address this issue here but to leave it for future studies and to focus on the numerical analysis of an approximation scheme. Assumption \ref{ass:sol} was made for this purpose.

 \item 
 The choice of the boundary conditions \eqref{bcond} has been made in order to keep the presentation as simple as possible. Prescribing non-zero Dirichlet boundary condition for $c$ does not change the analysis. For boundary data  $\ubc$ depending on time we also expect similar results. On the contrary, different conditions for $\kappa$ present difficulties as already noted and briefly discussed previously \cite[Remark ~2.3]{DD06}.

 \item 
 In \cite{DD06} a different choice for the initial values $\uich$ is made which improves the order of convergence: Let $\hat{u}_{0h}$ be given through \eqref{def-uhat} at time $t=0$, and $\hat{w}_{0h}$ through \eqref{(2.22)}. Define $\uich$ by $\uich - I_{h} \ubc \in X_{h0}$ and 
 \begin{align}\label{diffy}
  \int_{I} \frac{\uichx}{Q_{0h}} \varphi_{hx} dx = \int_{I} \frac{\hat{w}_{0h}}{\hat{Q}_{0h}} \varphi_{h} \qquad \forall \, \varphi_{h} \in X_{h0}. 
 \end{align}
 Here $Q_{0h}=\sqrt{1+ |\uichx|^{2}}$, $\hat{Q}_{0h}=\sqrt{1+ |\hat{u}_{0hx}|^{2}}$. Then for $e_{u}(0)= \hat{u}_{0h} - \uich$ and $e_{w}(0) = \hat{w}_{0h}- \wich$ we have the estimate 
 \begin{align*}
  \| e_{u}(0) \|_{H^{1}(I)} + \|e_{w}(0) \|_{L^{2}(I)} \leq C h^{2}|\log h|.
 \end{align*}
 The proof is sketched in \cite{DD06Corr}.  However, this choice of initial values is not effective in our analysis as that higher order is not achieved with regards to the other terms in our case of a coupled problem. 

 \item 
 Lemma~\ref{lemma3.4} corresponds to \cite[Lemma 3.4]{DD06} where the coefficient $1/4C_{0}$ has been corrected.
 
\end{itemize}

\bibliographystyle{siam}

\end{document}